\documentclass[12pt, reqno]{amsart}
\usepackage{amsmath,amssymb,amsthm}
\allowdisplaybreaks
\usepackage{amsmath, amssymb}
\usepackage{amsthm, amsfonts, mathrsfs}
\usepackage{mathptmx}
\usepackage{fullpage}
\usepackage{amsfonts,graphicx}
\usepackage{dutchcal}
\numberwithin{equation}{section}
\usepackage[colorlinks=true, pdfstartview=FitV, linkcolor=blue, citecolor=blue, urlcolor=blue]{hyperref}

\newtheorem {theorem}{Theorem}[section]
\newtheorem {lemma}[theorem]{{\bf Lemma}}

\newtheorem {proposition}[theorem]{{\bf Proposition}}
\theoremstyle{remark}
\newtheorem {remark}{{\bf Remark}}[section]

\theoremstyle{plain} \numberwithin {equation}{section}

\newcommand{\R}{{\mathbb R}}
\def\div{ \hbox{\rm div}\,  }

\def\nn{\nonumber}

\def\u{ \mathbf{u} }
\def\v{ \mathbf{v} }
\def\T{ \mathbb{T} }

\begin{document}
\title{ { Improved stability threshold for 2D Navier-Stokes    \\[1ex]Couette flow in an infinite channel }}

%\author{XXXXXXXXXXXXXXXX}
\author{Tao Liang}

\address[T. Liang]{ School of Mathematics,
	South China University of Technology,
	Guangzhou, 510640, China}
\email{taolmath@163.com}

\author[J. Wu]{Jiahong Wu}
\address[J. Wu]{Department of Mathematics, University of Notre
	Dame, Notre Dame, IN 46556, USA } \email{jwu29@nd.edu}

 \author[X. Zhai]{Xiaoping Zhai}
\address[X. Zhai]{School of Mathematics and Statistics, Guangdong University of Technology,
Guangzhou, 510520, China} \email{pingxiaozhai@163.com (Corresponding author)}

%\footnote{\today}



\begin{abstract}
 We study the nonlinear stability of the two-dimensional Navier-Stokes equations around the Couette shear flow in the channel domain $\mathbb{R}\times[-1,1]$ subject to Navier slip boundary conditions. We establish a quantitative stability threshold for perturbations of the initial vorticity $\omega_{in}$, showing that stability holds for perturbations of order $\nu^{1/2}$ measured in an anisotropic Sobolev space. This sharpens the recent work of Arbon and Bedrossian [Comm. Math. Phys., 406 (2025), Paper No. 129] who proved stability under the threshold $\nu^{1/2}(1+\ln(1/\nu))^{-1/2}$. Our result removes the logarithmic loss and  identifies the natural scaling $\nu^{1/2}$ as the critical size of perturbations for nonlinear stability in this setting.
\end{abstract}

\maketitle

\tableofcontents
\section{Introduction and the main result}
\subsection {Two dimensional Navier-Stokes  equations}

In this paper, we consider the two-dimensional Navier-Stokes equations in the infinite channel $(x,y) \in\mathbb{R} \times [-1, 1]$ with Navier slip boundary conditions, which is governed by the following model
\begin{eqnarray}\label{NS_eq}
\left\{\begin{aligned}
&\partial_t \v+ \v \cdot\nabla \v - \nu \Delta \v  + \nabla P  = 0   ,\\
& \div \v = \partial_x v_1 + \partial_y v_2 = 0,\\
& v_2(t, x, \pm 1) = 0, \quad \partial_y v_1(t, x, \pm 1) = 1,\\
& \v(0, x,y) = \v_{in}(x,y).
\end{aligned}\right.
\end{eqnarray}
Here, $ \mathbf{v}=(v_1,v_2) $ denotes the velocity field, $ P $ represents the pressure, and $ \nu $ signifies the viscosity coefficient of the fluid.  We investigate the stability threshold problem for \eqref{NS_eq}  with the background Couette flow $\mathbf{v}_s = (y, 0)^\top$. By setting $ \u = \v-\v_s$, the system can be reformulated as follows
\begin{eqnarray}\label{rewrite}
\left\{\begin{aligned}
&\partial_t \u+ y\partial_x \u + \binom{u_2  }{ 0} + \u \cdot\nabla \u - \nu \Delta \u  + \nabla P  =  0,\\
& \div \u =  0,\\
& u_2(t, x, \pm 1) = 0, \quad \partial_y u_1(t, x, \pm 1) = 0.
\end{aligned}\right.
\end{eqnarray}
By introducing the vorticity $ \omega = \nabla \times \u = \partial_y u_1 - \partial_x u_2$, we obtain the following system of equations for the vorticity
\begin{eqnarray}\label{rewrite1}
\left\{\begin{aligned}
&\partial_t \omega+ y\partial_x \omega   + \u \cdot\nabla \omega - \nu \Delta \omega    =  0,\\
& \omega(0, x, y) = \omega_{in}(x,y), \quad \omega(t, x, \pm 1) = 0.\\
& \u = \nabla^{\perp} \phi = (\partial_y \phi,-\partial_x \phi ), \quad \Delta \phi =\omega.
\end{aligned}\right.
\end{eqnarray}
\subsection{Backgrounds}
The stability of Couette flow has been a central theme in fluid mechanics since the pioneering works of Kelvin \cite{Kelvin}, Rayleigh \cite{Rayleigh}, Orr \cite{Orr1907}, and Sommerfeld \cite{Sommerfeld}. In his seminal paper \cite{Kelvin}, Kelvin derived an explicit solution to the linearized vorticity equation associated with \eqref{rewrite1}:
\begin{equation}
	\partial_t \omega + y \partial_x \omega - \nu \Delta \omega = 0,
	\quad \omega|_{t=0} = \omega_{in}.
\end{equation}
Let $\widehat{\omega}(t,k,\xi)$ denote the Fourier transform of $\omega(t,x,y)$. On the periodic channel $\mathbb{T}_x \times \mathbb{R}_y$, the Fourier frequency $k$ belongs to $\mathbb{Z}$, and Kelvin's explicit representation takes the form
\begin{equation}
	\widehat{\omega}(t,k,\xi)
	= \widehat{\omega}_{in}(k,\xi+kt)
	\exp\!\left(-\nu \int_0^t \big(|k|^2 + |\xi+k(t-s)|^2\big)\, ds\right).
\end{equation}
From this formula one derives two fundamental linear estimates:
\begin{align}
	&|\widehat{\omega}(t,k,\xi)|
	\leq C\, |\widehat{\omega}_{in}(k,\xi+kt)|
	\,e^{-c\nu^{1/3}|k|^{2/3}t}, \label{eq1.5a} \\[6pt]
	&|\widehat{\phi}(t,k,\xi)|
	\leq C\, \langle t\rangle^{-2}
	\frac{1+|k|^2+|\xi+kt|^2}{|k|^4}\,
	|\widehat{\omega}_{in}(k,\xi+kt)|\,
	e^{-c\nu^{1/3}|k|^{2/3}t}, \label{eq1.6}
\end{align}
where $\phi = \Delta^{-1}\omega$ denotes the stream function.

Estimate \eqref{eq1.5a} shows that the nonzero modes of $\omega$ decay on the time scale $t \gtrsim \nu^{-1/3}$, which is dramatically faster than the usual heat dissipation scale $t \gtrsim \nu^{-1}$. This accelerated decay mechanism is known as \emph{enhanced dissipation} \cite{albritton}. Meanwhile, estimate \eqref{eq1.6} demonstrates that $\phi$ decays at a polynomial rate, a phenomenon referred to as \emph{inviscid damping}, first observed by Orr in \cite{Orr1907}. Together, enhanced dissipation and inviscid damping form the core linear stability mechanisms underpinning the long-time dynamics of shear flows.

In the inviscid case ($\nu=0$), nonlinear inviscid damping was rigorously established for the Couette flow by Bedrossian and Masmoudi \cite{BM2015}, with further developments in \cite{IJ2018, rensiqi, wei2019, wei20192, zhuhao}. More recent works \cite{IJ2020, zhaoweiren2020} extended the theory to show that nonlinear inviscid damping also holds for general classes of monotone shear flows.

This paper addresses the stability problem for the Navier-Stokes equations with small viscosity $\nu > 0$. Mathematically, the transition threshold problem was formulated by Bedrossian, Germain, and Masmoudi \cite{BGM2017} as follows:
\vskip .1in
{\it
Given a norm $\|\cdot\|_X$, determine a value $\beta = \beta(X)$ such that
\begin{itemize}
\item$\|\mathbf{u}_{\text{in}}\|_{X} \leq \nu^\beta$ implies stability, enhanced dissipation, and inviscid damping;
\item $\|\mathbf{u}_{\text{in}}\|_{X} \gg \nu^\beta$ leads to instability,
\end{itemize}
where the exponent $\beta=\beta (X)>0$  is called the \textit{stability threshold}. And it is also known as the \textit{transition threshold} in almost all applied literature. Obviously, the \textit{transition threshold} problem is more stringent and complicated than the nonlinear stability problem.
}

\vskip .1in
On the domain $\T_x\times\R_y$, the following important results are known:
\begin{itemize}
	\item If $X$ is Gevrey class $2_-$, then \cite{BMV2016} showed $\beta\leq 0$, while \cite{DM2023} established $\beta\geq 0$.
	\item If $X$ is the Sobolev space $H^{\log}_xL^2_y$, then \cite{BVW2018,zhaoweiren2020cpde} obtained $\beta\leq \tfrac12$, and \cite{LiMasmoudiZhao2022critical} proved $\beta\geq \tfrac12$.
	\item If $X$ is the Sobolev space $H^\sigma$ with $\sigma\geq 2$, then \cite{zhaoweiren2019,wei2023} showed $\beta\leq \tfrac13$.
	\item If $X$ is Gevrey class $\tfrac1s$ with $s\in[0,\tfrac12]$, then \cite{LMZ2022G} proved $\beta\leq \tfrac{1-2s}{3(1-s)}$.
\end{itemize}

In addition, significant progress has been made on the stability threshold of the 2D Couette flow in a finite channel (see, e.g., \cite{BHIW2023,BHIW,CLWZ2020}), and several results are also available for the 3D Couette flow (see, e.g., \cite{BGM2017,BGM2020,BGM2015,Chenwei2020,wei2020}).

More recently, fully nonlinear stability results have been obtained for Couette flow in unbounded, non-periodic domains (see \cite{arbon2024,liuning,wangweike}).
In particular, Arbon--Bedrossian \cite{arbon2024} established a quantitative stability threshold for the 2D Navier--Stokes Couette flow on the plane, the half-plane with Navier boundary conditions, and the infinite channel with Navier boundary conditions.
They proved that perturbations of size
\[
O\!\left(\frac{\nu^{1/2}}{\sqrt{\ln(1/\nu)}}\right)
\]
remain globally stable and exhibit inviscid damping of the velocity field, enhanced dissipation at intermediate and high frequencies, and dispersion phenomena (Taylor dispersion on the plane and half-plane, heat-like decay in the channel).
These are the first fully nonlinear stability results for Couette flow in unbounded, non-periodic domains, extending prior work that focused primarily on periodic settings.
Their work also introduced new techniques to control low- and intermediate-frequency interactions in the nonlinear regime, providing a framework that may serve as a stepping stone toward the 3D case and more general shear flows.

Utilizing more detailed frequency space decomposition techniques, the goal of the present paper is to improve upon the stability threshold derived in Theorem 1.3 of Arbon--Bedrossian \cite{arbon2024}.
More precisely, we establish the following main theorem.

\subsection{Main result}
\begin{theorem}\label{thm1.2} (Nonlinear stability)
	Let $\varepsilon \in (0, \frac{1}{12})$ and consider System \eqref{rewrite1} under the boundary condition
	\[
	\omega_{in}|_{y = \pm 1} = 0.
	\]
	Then, for any $m \in (1, \infty)$, there exists a small positive constant $\varepsilon_0$ such that, if the initial data satisfies
	\[
	\sum_{j=0}^{1} \left\| (\nu^{\frac{1}{3}} \partial_y)^j \langle \partial_x \rangle^{\, m - \frac{j}{3}} \left\langle \frac{1}{\partial_x} \right\rangle^{\varepsilon} \omega_{in} \right\|_{L^2_{x,y}} \le \varepsilon_0 \, \nu^{\frac{1}{2}},
	\]
	then, for any sufficiently small positive constant $\delta$ (independent of $\nu$) and for all $\nu \in (0,1)$, System \eqref{rewrite1} admits a globally well-posed solution $\omega$ satisfying the uniform stability estimates
	\[
	\sum_{j=0}^{1} \left\| e^{\delta \lambda_k t} (\nu^{\frac{1}{3}} \partial_y)^j \langle \partial_x \rangle^{\, m - \frac{j}{3}} \left\langle \frac{1}{\partial_x} \right\rangle^{\varepsilon} \omega \right\|_{L^2_{x,y}}
	\le C \, \varepsilon_0 \, \nu^{\frac{1}{2}}, \quad \text{where} \quad
	\lambda_k =
	\begin{cases}
		\nu^{\frac{1}{3}} |k|^{\frac{2}{3}}, & |k| \ge \nu, \\[2mm]
		\nu, & |k| \le \nu.
	\end{cases}
	\]
	for any $t \in [0, +\infty)$. Moreover, the solution exhibits the following inviscid damping:
	\[
	\sum_{j=0}^{1} \left\| e^{\delta \lambda_k t} \langle \partial_x \rangle^{\, m - \frac{j}{3}} \left\langle \frac{1}{\partial_x} \right\rangle^{\varepsilon} (\nu^{\frac{1}{3}} \partial_y)^j \partial_x \u \right\|_{L_t^2 L_{x,y}^2}
	\le C \, \varepsilon_0 \, \nu^{\frac{1}{2}}.
	\]
\end{theorem}
\begin{remark}
In the work of Arbon--Bedrossian \cite{arbon2024}, the nonlinear stability threshold for the
two-dimensional Navier--Stokes Couette flow in an infinite channel was obtained at the scale
$$
\nu^{1/2}(1+\ln(1/\nu))^{-1/2},
$$
which may not be optimal. The logarithmic loss originates from the
difficulty of controlling low-frequency interactions in the continuous Fourier spectrum.
In particular, near the zero mode ($k \approx 0$), nonlinear terms may accumulate
growth that is hard to suppress by standard energy methods. To address this,
Arbon--Bedrossian employed a frequency partitioning strategy together with commutator
estimates, but the resulting control may not be sharp, leading to the extraneous logarithmic factor.

In the present paper, we refine this analysis by introducing a new weighted energy
framework that is specifically adapted to the infinite channel setting (see  $E$ and $D$ in
Section~3).
More precisely, we construct anisotropic Sobolev norms with frequency weights
that distinguish between the high-, intermediate-, and low-frequency regimes.
The associated energy functional captures both enhanced dissipation at high
frequencies (with time scale $t \sim \nu^{-1/3}$) and inviscid damping at low and
intermediate frequencies. By carefully tracking the evolution of these weighted
energies, we obtain sharper bounds on the stream function and vorticity
interactions and thereby eliminate the logarithmic loss.

As a consequence, we establish the nonlinear stability threshold at the
optimal scale $\nu^{1/2}$,
which matches the prediction from hydrodynamic stability theory.
In other words, perturbations of size $O(\nu^{1/2})$ remain globally stable and
exhibit enhanced dissipation and inviscid damping, while larger perturbations
may lead to instability.
\end{remark}

\vskip .1in
\subsection*{Notations}
Throughout this paper, $C > 0$ denotes a generic constant that is independent of the relevant quantities. For brevity, we write $f \lesssim g$ to mean that $f \le C g$ for some constant $C > 0$.
For two operators $A$ and $B$, we denote their commutator by $[A, B] = AB - BA$.
Moreover, we use $\langle f, g \rangle$ to represent the $L^2_y([-1, 1])$ inner product of $f$ and $g$.

\section{Linear Stability}\label{sec2}

\subsection{Linear Estimates}
Before proceeding with the proof of the theorem, we introduce some notational definitions. For $(x, y) \in \mathbb{R} \times [-1, 1]$, we define the Fourier transform of $ f $ with respect to the $ x $-direction as $ f_k $, given by:
$$ f_k(t,y) = \frac{1}{2\pi} \int_{\R} f(t,x,y)e^{-ikx} {d}x.$$
To investigate the linearized stability of System \eqref{rewrite1}, we first apply the Fourier transform with respect to the $ x $-direction, yielding the following equation for the vorticity $ \omega $ in terms of the wave number $ k $:
\begin{eqnarray}\label{linear_eq}
\left\{\begin{aligned}
&\partial_t \omega_k+ ik y \omega_k   - \nu (-  k^2 + \partial_y^2) \omega_k    =  0,\\
&    \Delta_{k} \phi_k =(-k^2 + \partial_y^2)\phi_k=\omega_k,\\
& \phi_k (t, \pm 1) = \omega_k(t, \pm 1) = 0.
\end{aligned}\right.
\end{eqnarray}
Inspired by the approach in \cite{BHIW2023, BHIW20251} for studying the stability threshold problem of System \eqref{rewrite1} in a finite channel $ \T \times [-1,1]$, where a singular integral operator $ \mathfrak{J}_k $ (to be determined) is constructed to capture the inviscid damping effect. However, we cannot directly apply this method to the infinite channel domain $ \R \times [-1,1]$. This is because, in the high-frequency regime, the influence of the nonlinear terms is relatively minor, as higher-frequency perturbations typically experience stronger dissipative effects. In contrast, in the low-frequency regime, the nonlinear terms exert a stronger influence, potentially leading to excessive growth of these terms. When $ x \in \mathbb{T} $, the frequencies associated with $ \Delta_x $ are discrete; for nonzero modes, there exists a gap between frequencies $ k $ and $ k+1 $, which facilitates the capture of enhanced dissipation. For the zero mode, however, no such enhanced dissipation occurs. When considering $ x \in \mathbb{R} $, the frequencies become continuous, and the proximity to the zero mode may induce excessive growth in the nonlinear terms, rendering energy estimates particularly challenging. Inspired by \cite{arbon2024}, we address this by appropriately partitioning the frequency space. To this end,  we define the following coercive energy functional:
\begin{align*}%\label{Ekw}
E_{ k}[\omega_k] \stackrel{\mathrm{def}}{=}\begin{cases}\| \omega_k\|^2_{L^2} + c_{\alpha} \alpha   \|\partial_y \omega_k \|^2_{L^2} -c_{\beta} \beta
\operatorname{Re} \langle i k \omega_k, \partial_y \omega_k \rangle + c_{\tau} \operatorname{Re} \langle  \omega_k, \mathfrak{J}_k [\omega_k] \rangle & |k| \geq \nu ;\\ \qquad \quad  + c_{\tau} c_{\alpha} \alpha \operatorname{Re} \langle \mathfrak{J}_k [\partial_{y} \omega_k], \partial_{y} \omega_{k} \rangle,  \\ \| \omega_k\|^2_{L^2} + c_{\alpha} \alpha   \|\partial_y \omega_k \|^2_{L^2}  + c_{\tau} \operatorname{Re} \langle  \omega_k, \mathfrak{J}_k [\omega_k] \rangle  + c_{\tau} c_{\alpha} \alpha \operatorname{Re} \langle \mathfrak{J}_k [\partial_{y} \omega_k], \partial_{y} \omega_{k} \rangle, & |k| \leq \nu.\end{cases}
\end{align*}
Here, $ \alpha $, $ \beta $, and $ \lambda_k $ are defined as follows:
\begin{align*}
\alpha \stackrel{\mathrm{def}}{=}\begin{cases} \nu^{\frac{2}{3}} |k|^{-\frac{2}{3}} , & |k| \geq \nu ; \\ 1, & |k| \leq \nu;\end{cases}, \quad \beta\stackrel{\mathrm{def}}{=} \begin{cases} \nu^{\frac{1}{3}} |k|^{-\frac{4}{3}} , & |k| \geq \nu ; \\ 0, & |k| \leq \nu;\end{cases}
\quad\lambda_k \stackrel{\mathrm{def}}{=} \begin{cases} \nu^{\frac{1}{3}} |k|^{\frac{2}{3}} , & |k| \geq \nu ; \\ \nu, & |k| \leq \nu.\end{cases}
\end{align*}
It is straightforward to verify that the energy functional satisfies $$ E_{ k}[\omega_k] \approx \| \omega_k\|^2_{L^2} +    \alpha_k\|\partial_y \omega_k \|^2_{L^2} .$$
The singular integral operator $ \mathfrak{J}_k$ is defined as follows
\begin{align*}%\label{Jk}
\mathfrak{J}_k[f](y) \stackrel{\mathrm{def}}{=}|k| \mathrm{p} . \mathrm{v} . \frac{k}{|k|} \int_{-1}^1 \frac{1}{2 i(y-y^{\prime})} G_k(y, y^{\prime}) f(y^{\prime}) {d} y^{\prime},
\end{align*}
where $G_k(y, y')$ represents the Green's function satisfying $ \Delta_{k} G_k (y, y') = \delta(y-y')$ with homogeneous Dirichlet boundary conditions. The explicit expression for $ G_k(y, y')$ is given by
\begin{align*}%\label{Gk}
G_k\left(y, y^{\prime}\right)\stackrel{\mathrm{def}}{=}-\frac{1}{k \sinh (2 k)} \begin{cases}\sinh \left(k\left(1-y^{\prime}\right)\right) \sinh (k(1+y)), & y \leq y^{\prime} ; \\ \sinh (k(1-y)) \sinh \left(k\left(1+y^{\prime}\right)\right), & y \geq y^{\prime}.\end{cases}
\end{align*}

We briefly review the boundness of the singular integral operator $ \mathfrak{J}_k$ with $k \neq 0$  introduced in \cite{BHIW2023}, the estimate for the commutator $ [\partial_y, \mathfrak{J}_k] $, and the conjugate symmetry of $ \mathfrak{J}_k $, as encapsulated in the three lemmas presented below. The detailed proofs can be found in \cite{arbon2024,BHIW2023}, and we omit the proof details here.

\begin{lemma}
\label{le2.1}
The singular integral operator $ \mathfrak{J}_k $ is a bounded linear operator from $ L^2 $ to $ L^2 $, and furthermore,
\begin{align*}
	\left\|\mathfrak{J}_k\right\|_{L^2 \rightarrow L^2} \lesssim 1.
\end{align*}
\end{lemma}
\begin{lemma}\label{le2.2}
	For the commutator $[\partial_y, \mathfrak{J}_k]$, there holds
	\begin{align*}
		\| [\partial_y, \mathfrak{J}_k] \|_{L^2 \rightarrow L^2} \lesssim |k|.
	\end{align*}
\end{lemma}
\begin{lemma}\label{le2.3}
	For all $f, g \in L^2$ there holds
	\begin{align*}
		\overline{\mathfrak{J}_k[f]}=-\mathfrak{J}_k[\bar{f}],
	\end{align*}
	and
	\begin{align*}
	\int_{-1}^1 \bar{f} \mathfrak{J}_k[g] {d} y=-\int_{-1}^1 \mathfrak{J}_k[\bar{f}] g {d} y = \int_{-1}^1 \overline{\mathfrak{J}_k[f]} g {d} y .
	\end{align*}
	In particular, we have $\mathfrak{J}_k=\mathfrak{J}_k^*$.
\end{lemma}

Additionally, based on the energy functional defined above, we define the corresponding dissipative energy functional
\begin{align*}
&\operatorname{Dis}_{1} \stackrel{\mathrm{def}}{=} \nu \| \nabla_k \omega_k\|^2_{L^2}, \quad \operatorname{Dis}_{2} \stackrel{\mathrm{def}}{=} \nu \alpha \| \partial_{y} \nabla_k \omega_k\|^2_{L^2},\quad\operatorname{Dis}_{3}\stackrel{\mathrm{def}}{=} \lambda_k \|  \omega_k\|^2_{L^2},\nn\\
& \operatorname{Dis}_{4} \stackrel{\mathrm{def}}{=} |k|^2 \| \nabla_k \phi_k\|^2_{L^2},
\quad\operatorname{Dis}_{5} \stackrel{\mathrm{def}}{=} \alpha |k|^2 \| \partial_{y}\nabla_k \phi_k\|^2_{L^2} \quad (\text{where} \nabla_k = (ik, \partial_{y})^\top ),
\end{align*}
associated with $ D_k[\omega_k] = \sum_{i=1}^{5} \operatorname{Dis}_i$.

As the stability of the linearized system \eqref{linear_eq} has been thoroughly investigated in \cite{arbon2024}, we therefore state only the final result and omit the detailed proof.
\begin{proposition}(see \cite{arbon2024})\label{linear_es}
	There exists a constant $c_0 = c_0\left(c_\tau, c_\alpha, c_\beta\right)$, which can be chosen independently of $v$ such that, for any $H^1$ solution $\omega_k$ to System \eqref{linear_eq}, the following holds for any $k \neq 0$:
	$$
	\frac{d}{d t} E_k\left[\omega_k\right]+c_0 D_k\left[\omega_k\right]+c_0 \lambda_k E_k\left[\omega_k\right]\le0.
	$$
\end{proposition}

\section{Nonlinear stability}
In this section, we prove the nonlinear stability stated in Theorem \ref{thm1.2}. The proof is divided into two subsections. The first subsection focuses on the construction of certain energy functionals and the establishment of preliminary lemmas concerning the stream function $\phi$. The second subsection employs a weighted energy method in space to control the potential growth of nonlinear terms, thereby closing the energy estimates.

\subsection{Notations and Preliminary Lemmas}In this subsection, corresponding to the energy functionals $ E_k\left[\omega_k\right]$ and $ D_k\left[\omega_k\right]$ defined above for the $k$-mode, we introduce the following weighted spatial energy functionals:
$$ \mathscr{E}\stackrel{\mathrm{def}}{=}  \int_{\R} e^{2c\lambda_k t} \langle k \rangle^{2m} \langle k^{-1} \rangle^{2\varepsilon}  E_k\left[\omega_k\right] \,{dk} ,$$
 and
 $$  \mathscr{D} \stackrel{\mathrm{def}}{=} \int_{\R} e^{2c\lambda_k t} \langle k \rangle^{2m} \langle{k}^{-1} \rangle^{2\varepsilon}  D_k\left[\omega_k\right] \,{dk} , \quad  \mathscr{D}_{i}\stackrel{\mathrm{def}}{=} \int_{\R} e^{2c\lambda_k t} \langle k \rangle^{2m} \langle{k}^{-1} \rangle^{2\varepsilon} \operatorname{Dis}_{ i} \,{dk}. $$

\begin{lemma}\label{pre1}
	There hold the following estimates:
	\begin{align*}
		&\left\|  \| \partial_{y} \phi_{\ell}\|_{L^\infty_y }\right\|_{L^1_\ell}\lesssim \mathscr{E}^{\frac{1}{2}},\quad  \left\| |\ell| \big\| \phi_{\ell}\big\|_{L^\infty_y }\right\|_{L^1_\ell}\lesssim \mathscr{D}_4^{\frac{1}{2}},\quad\left\| \big\| |k|^{-\frac{1}{2}}\partial_{y} \omega_k \big\|_{L^2_y } \right\|_{L^1_k} \lesssim \nu^{-\frac{1}{2}} \mathscr{D}_1^{\frac{1}{2}}, \nn\\
		&  \left\|   \langle{k}^{-1} \rangle^{\frac{1}{2}} \| \omega_k \|_{L^2_y}  \right\|_{L^1_k} \lesssim \mathscr{E}^{\frac{1}{2}},\quad \left\| \|  \omega_k \|_{L^\infty_y }\right\|_{L^1_k} \lesssim \nu^{-\frac{1}{2}} \mathscr{D}_1^{\frac{1}{2}}.
	\end{align*}
\end{lemma}
	\begin{proof}
		Concerning the first inequality, we first make use of the Gagliardo-Nirenberg inequality $ \| f\|_{L^\infty_y} \lesssim |\ell|^{-\frac{1}{2}} \| \nabla_\ell f\|_{L^2_y}$, followed by an application of H\"older's inequality, which yields
		\begin{align*}
			\left\|  \| \partial_{y} \phi_{\ell}\|_{L^\infty_y }\right\|_{L^1_\ell}
\lesssim& \left\|  \| |\ell|^{-\frac{1}{2}} \partial_{y} \nabla_\ell \phi_{\ell}\|_{L^2_y }\right\|_{L^1_\ell} \nn\\
\lesssim& \left\| e^{c \lambda_\ell t} \langle \ell \rangle^{m} \langle{\ell}^{-1} \rangle^{\varepsilon}  \|  \omega_{\ell} \|_{L^2_y}  \right\|_{L^2_\ell} \left\|  \langle \ell \rangle^{-m} \langle{\ell}^{-1} \rangle^{-\varepsilon} |\ell|^{-\frac{1}{2}} \right\|_{L^2_\ell}\nn\\
			\lesssim& \mathscr{E}^{\frac{1}{2}}.
		\end{align*}
		Similarly, one can get
		\begin{align*}
		\left\| |\ell| \| \phi_{\ell}\|_{L^\infty_y }\right\|_{L^1_\ell}
&\lesssim \left\| |\ell|^{\frac{1}{2}} \| \nabla_\ell \phi_{\ell}\|_{L^2_y }\right\|_{L^1_\ell} \nn\\
&\lesssim \left\| e^{c \lambda_\ell t} \langle \ell \rangle^{m} \langle{\ell}^{-1} \rangle^{\varepsilon} |\ell| \| \nabla_\ell \phi_{\ell} \|_{L^2_y}  \right\|_{L^2_\ell} \left\|  \langle \ell \rangle^{-m} \langle{\ell}^{-1} \rangle^{-\varepsilon} |\ell|^{-\frac{1}{2}} \right\|_{L^2_\ell}\nn\\
		&\lesssim \mathscr{D}_4^{\frac{1}{2}},
		\end{align*}
		and
		\begin{align*}
		\left\| \|  \omega_k \|_{L^\infty_y }\right\|_{L^1_k} \lesssim \left\| |k|^{-\frac{1}{2}}\|  \nabla_k \omega_k \|_{L^2_y }\right\|_{L^1_k} \lesssim \nu^{-\frac{1}{2}} \mathscr{D}_1^{\frac{1}{2}} \|\langle  k \rangle^{-m} \langle{k}^{-1} \rangle^{- \varepsilon} |k|^{-\frac{1}{2}} \|_{L^2_k} \lesssim \nu^{-\frac{1}{2}} \mathscr{D}_1^{\frac{1}{2}}.
		\end{align*}
		Furthermore, H\"older's inequality yields directly
		\begin{align*}
			\left\|   \langle{k}^{-1} \rangle^{\frac{1}{2}} \| \omega_k \|_{L^2_y} \right\|_{L^1_k} \lesssim& \mathscr{E}^{\frac{1}{2}} \| \langle  k \rangle^{-m} \langle{k}^{-1} \rangle^{\frac{1}{2}- \varepsilon}\|_{L^2_k} \lesssim \mathscr{E}^{\frac{1}{2}},\quad
\left\| \big\| |k|^{-\frac{1}{2}}\partial_{y} \omega_k \big\|_{L^2_y } \right\|_{L^1_k} \lesssim \nu^{-\frac{1}{2}} \mathscr{D}_1^{\frac{1}{2}}.
		\end{align*}	
This completes the proof of Lemma \ref{pre1}.
	\end{proof}

\begin{lemma}\label{pre2}
	There hold the following estimates:
	\begin{align*}
		&\left\|  e^{c \lambda_k t}  \sqrt{\alpha(k)}   \langle k \rangle^{m} \langle{k}^{-1} \rangle^{\varepsilon} |k|  \|   \omega_k\|_{L^2_y} \right\|^2_{L^2_k} \lesssim \mathscr{D}^{\frac{1}{2}}_{1} \mathscr{D}^{\frac{1}{2}}_{3},\\
		&\int_{\{|k| \ge \nu\}} e^{2c \lambda_k t} |k|^{\frac{2}{3}}  \langle k \rangle^{2m} \langle{k}^{-1} \rangle^{2\varepsilon} \| \omega_{k} \|^2_{L^2_y} {d} k \lesssim \nu^{-\frac{1}{3}}  \mathscr{E}^{\frac{1}{2}} \mathscr{D}_1^{\frac{1}{4}} \mathscr{D}_3^{\frac{1}{4}},\\
		&\left\| \nu^{\frac{1}{12}} |k|^{\frac{1}{6}} e^{c \lambda_k t} \langle k \rangle^m \langle{k}^{-1} \rangle^{\varepsilon} \| \omega \|_{L^2_y}\right\|_{L^2_{\{|k| \ge \nu\}}} \lesssim \mathscr{E}^{\frac{1}{4}} \mathscr{D}_3^{\frac{1}{4}},\\
		&\left\|e^{c \lambda_k t} \langle k \rangle^{m} \langle {k}^{-1} \rangle^{\varepsilon}  |k|^{\frac{1}{2}} \| \omega_{k}\|_{L^2_y}    \right\|_{L^2_{\{|k|\ge \nu\}}} \lesssim \nu^{-\frac{1}{4}} \mathscr{D}_1^{\frac{1}{8}} \mathscr{D}_3^{\frac{3}{8}}.
	\end{align*}
	\begin{proof}
		For the first inequality, we decompose the frequency domain into high and low parts and estimate them separately. Specifically, we write:
		\begin{align*}
			&\left\|  e^{c \lambda_k t}  \sqrt{\alpha(k)}   \langle k \rangle^{m} \langle{k}^{-1} \rangle^{\varepsilon}  |k| \|   \omega_k\|_{L^2_y} \right\|^2_{L^2_k}\nn\\
 &\quad= \left( \int_{\{|k| \le \nu\}} + \int_{\{|k| \ge \nu\}} \right) e^{2c \lambda_k t}  \alpha(k)   \langle k \rangle^{2m} \langle{k}^{-1} \rangle^{2\varepsilon} |k|^2  \|   \omega_k\|^2_{L^2_y} \,{dk}.
		\end{align*}
		For the low-frequency part, using the condition $ |k| \le \nu$ and H\"older's inequality, we obtain
		\begin{align*}
			& \int_{\{|k| \le \nu\}} e^{2c \lambda_k t}  \alpha(k)   \langle k \rangle^{2m} \langle{k}^{-1} \rangle^{2\varepsilon} |k|^2  \|   \omega_k\|^2_{L^2_y} \,{dk}\nn\\
&\quad\lesssim \nu \int_{\{|k| \le \nu\}} e^{2c \lambda_k t}  \   \langle k \rangle^{2m} \langle{k}^{-1} \rangle^{2\varepsilon} |k|  \|   \omega_k\|^2_{L^2_y} \,{dk}\nn\\
			&\quad\lesssim \nu^{\frac{1}{2}} \left\| e^{c \lambda_k t}     \langle k \rangle^{m} \langle{k}^{-1} \rangle^{\varepsilon} |k|  \|   \omega_k\|_{L^2_y}\right\|_{L^2_{\{|k| \le \nu\}}}  \nu^{\frac{1}{2}} \left\| e^{c \lambda_k t}     \langle k \rangle^{m} \langle{k}^{-1} \rangle^{\varepsilon}  \|   \omega_k\|_{L^2_y}\right\|_{L^2_{\{|k| \le \nu\}}}\nn\\
 &\quad\lesssim \mathscr{D}^{\frac{1}{2}}_{1} \mathscr{D}^{\frac{1}{2}}_{3}.
		\end{align*}
		Similarly, the high-frequency part can be estimated as follows:
		\begin{align*}
			\int_{\{|k| \ge \nu\}} e^{2c \lambda_k t}  \alpha(k)   \langle k \rangle^{2m} \langle{k}^{-1} \rangle^{2\varepsilon} |k|^2  \|   \omega_k\|^2_{L^2_y} \,{dk} &\lesssim  \int_{\{|k| \ge \nu\}} \nu^{\frac{2}{3}} |k|^{\frac{4}{3}} e^{2c \lambda_k t}  \   \langle k \rangle^{2m} \langle{k}^{-1} \rangle^{2\varepsilon}  \|   \omega_k\|^2_{L^2_y} \,{dk}\nn\\
			&\lesssim \mathscr{D}^{\frac{1}{2}}_{1} \mathscr{D}^{\frac{1}{2}}_{3}.
		\end{align*}
		To prove the second and fourth inequalities, we invoke the definition of $ \mathscr{D}_i$ and apply H\"older's inequality
		\begin{align*}
			&\int_{\{|k| \ge \nu\}} e^{2c \lambda_k t} |k|^{\frac{2}{3}}  \langle k \rangle^{2m} \langle{k}^{-1} \rangle^{2\varepsilon} \| \omega_{k} \|^2_{L^2_y} {d} k\\
			&\quad \lesssim \int_{\{|k| \ge \nu\}} e^{c \lambda_k t}   \langle k \rangle^{m} \langle{k}^{-1} \rangle^{\varepsilon} \| \omega_{k} \|_{L^2_y} \cdot e^{\frac{c}{2} \lambda_k t}   \langle k \rangle^{\frac{m}{2}} \langle{k}^{-1} \rangle^{\frac{\varepsilon}{2}} |k|^{\frac{1}{2}} \| \omega_{k} \|^{\frac{1}{2}}_{L^2_y} \cdot e^{\frac{c}{2} \lambda_k t}   \langle k \rangle^{\frac{m}{2}} \langle{k}^{-1} \rangle^{\frac{\varepsilon}{2}} |k|^{\frac{1}{6}}\| \omega_{k} \|^{\frac{1}{2}}_{L^2_y}  {d} k\\
			&\quad\lesssim \nu^{-\frac{1}{3}}  \mathscr{E}^{\frac{1}{2}} \mathscr{D}_1^{\frac{1}{4}} \mathscr{D}_3^{\frac{1}{4}},
		\end{align*}
		and
		\begin{align*}
		&\int_{\{|k| \ge \nu\}} e^{2c \lambda_k t} \langle k \rangle^{2m} \langle\frac{1}{k} \rangle^{2\varepsilon}  |k| \| \omega_{k}\|^2_{L^2_y}    {d} k \nn\\
&\quad\lesssim \left\| \langle k \rangle^{\frac{3m}{2}} \langle\frac{1}{k} \rangle^{\frac{3\varepsilon}{2}} |k|^{\frac{1}{2}} \|\omega_{k} \|^{\frac{3}{2}}_{L^2_y}\right\|_{L^\frac{4}{3}_{\{|k| \ge \nu\}}}  \left\| \langle k \rangle^{\frac{m}{2}} \langle\frac{1}{k} \rangle^{\frac{\varepsilon}{2}} |k|^{\frac{1}{2}} \|\omega_{k} \|^{\frac{1}{2}}_{L^2_y}\right\|_{L^4_{\{|k| \ge \nu\}}}\\
		&\quad\lesssim \nu^{-\frac{1}{2}} \mathscr{D}_1^{\frac{1}{4}} \mathscr{D}_3^{\frac{3}{4}}.
		\end{align*}
The third inequality follows directly from an application of  H\"older's inequality. This completes the proof of Lemma \ref{pre2}.

	\end{proof}
\end{lemma}

\begin{lemma}\label{pre3}
	There hold the following estimates:
	\begin{align*}
		\left\|  e^{c \lambda_\ell t}     \langle \ell \rangle^{m} \langle{\ell}^{-1} \rangle^{\varepsilon} |\ell|^{\frac{7}{6}}  \| \partial_{y}  \phi_\ell\|_{L^\infty_y} \right\|^2_{L^2_{\{|\ell| \ge \nu\}}}
\lesssim& \nu^{-\frac{2}{3}}\mathscr{D}_{5};\\
		\left\|   \| \partial^2_{y}  \phi_\ell\|_{L^\infty_y} \right\|_{L^1_\ell}  \lesssim& \nu^{-\frac{1}{4}} \mathscr{E}^{\frac{1}{4}} \mathscr{D}^{\frac{1}{4}}_{1};\\
		\left\|\nu^{\frac{1}{3}} |\ell|^{\frac{2}{3}} \langle \ell \rangle^{m} \langle{\ell}^{-1} \rangle^{\varepsilon}  \| \partial^2_{y} \phi_{\ell} \|_{L^\infty_y}\right\|_{L^2_\ell} \lesssim& \nu^{-\frac{1}{4}} \mathscr{D}_2^{\frac{1}{4}} \mathscr{D}_5^{\frac{1}{4}}.
	\end{align*}
	\begin{proof}
		To prove the first inequality, we employ the Gagliardo-Nirenberg inequality, which yields
		\begin{align*}
			\left\|  e^{c \lambda_\ell t}     \langle \ell \rangle^{m} \langle{\ell}^{-1} \rangle^{\varepsilon} |\ell|^{\frac{7}{6}}  \| \partial_{y}  \phi_\ell\|_{L^\infty_y} \right\|^2_{L^2_{\{|\ell| \ge \nu\}}}
\lesssim& \int_{\{|\ell| \ge \nu\}} e^{2c \lambda_\ell t}     \langle \ell \rangle^{2m} \langle{\ell}^{-1} \rangle^{2\varepsilon} |\ell|^{\frac{4}{3}}  \| \partial_{y} \nabla_\ell \phi_\ell\|^2_{L^2_y}
\,{d\ell}\nn\\
\lesssim& \nu^{-\frac{2}{3}}\mathscr{D}_{5}.
		\end{align*}
		In addition, the second and third inequalities can be obtained by using the Gagliardo-Nirenberg inequality together with H\"older's inequality, giving
		\begin{align*}
			\left\|   \| \partial^2_{y}  \phi_\ell\|_{L^\infty_y} \right\|_{L^1_\ell} &\lesssim \int_{\R} \| \partial_y^2 \phi_{\ell}\|^{\frac{1}{2}}_{L^2_y} \| \partial_y^3 \phi_{\ell}\|^{\frac{1}{2}}_{L^2_y}  {d\ell} \\
			& \lesssim \int_{\R} \langle \ell \rangle^{-m} \langle{\ell}^{-1} \rangle^{-\varepsilon} \left(\langle \ell \rangle^{\frac{m}{2}} \langle{\ell}^{-1} \rangle^{\frac{\varepsilon}{2}} \| \partial_y^2 \phi_{\ell}\|^{\frac{1}{2}}_{L^2_y}\right) \left(\langle \ell \rangle^{\frac{m}{2}} \langle{\ell}^{-1} \rangle^{\frac{\varepsilon}{2}} \| \partial_y^3 \phi_{\ell}\|^{\frac{1}{2}}_{L^2_y}\right)  {d\ell}\\
			&\lesssim \left\| \langle \ell \rangle^{-m} \langle{\ell}^{-1} \rangle^{-\varepsilon}\right\|_{L^2_\ell} \left\| \langle \ell \rangle^{m} \langle{\ell}^{-1} \rangle^{\varepsilon} \| \partial_y^2 \phi_{\ell}\|_{L^2_y}\right\|^{\frac{1}{2}}_{L^2_\ell} \left\| \langle \ell \rangle^{m} \langle{\ell}^{-1} \rangle^{\varepsilon} \| \partial_y \omega_{\ell}\|_{L^2_y}\right\|^{\frac{1}{2}}_{L^2_\ell} \\
			&\lesssim \nu^{-\frac{1}{4}} \mathscr{E}^{\frac{1}{4}} \mathscr{D}^{\frac{1}{4}}_{1} ,
		\end{align*}
		and
		\begin{align*}
			& \int_{\{|\ell| \ge \nu\}} \nu^{\frac{2}{3}} |\ell|^{\frac{4}{3}} \langle \ell \rangle^{2m} \langle{\ell}^{-1} \rangle^{2\varepsilon}  \| \partial^2_{y} \phi_{\ell} \|^2_{L^\infty_y} \,{d\ell}\\
			 &\quad\lesssim \int_{\{|\ell| \ge \nu\}} \nu^{\frac{1}{3}} |\ell|^{-\frac{1}{3}} \langle \ell \rangle^{m} \langle{\ell}^{-1} \rangle^{\varepsilon} |\ell| \| \partial^3_{y} \phi_{\ell} \|_{L^2_y} \cdot \nu^{\frac{1}{3}} |\ell|^{-\frac{1}{3}} \langle \ell \rangle^{m} \langle{\ell}^{-1} \rangle^{\varepsilon} |\ell| \| \partial^2_{y} \phi_{\ell} \|_{L^2_y} \,{d\ell}\\
			 &\quad\lesssim \nu^{-\frac{1}{2}} \mathscr{D}_5^{\frac{1}{2}} \mathscr{D}_2^{\frac{1}{2}}.
		\end{align*}
This completes the proof of Lemma \ref{pre3}.
	\end{proof}
\end{lemma}

\subsection{Nonlinear Estimates } In this subsection, we primarily focus on controlling the potential growth of nonlinear terms. For convenience, we consider the equations of System \eqref{rewrite1} in the \( k \)-mode, specifically
\begin{eqnarray}\label{nonlinear_eq}
\left\{\begin{aligned}
&\partial_t \omega_k  = \mathcal{L}_{\omega, k} + \mathcal{N}_{\omega, k} ,\\
&    \Delta_{k} \phi_k =(-k^2 + \partial_y^2)\phi_k=\omega_k,\\
& \phi_k (t, \pm 1) = \omega_k(t, \pm 1) = 0,
\end{aligned}\right.
\end{eqnarray}
where
\begin{align*}
	\mathcal{L}_{\omega, k}\stackrel{\mathrm{def}}{=} -iky \omega_k   + \nu\Delta_{k}\omega_k , \quad
	 \mathcal{N}_{\omega, k}\stackrel{\mathrm{def}}{=}-(\nabla^{\top} \phi \cdot \nabla \omega)_k.
\end{align*}
Based on the definition of the energy functional $ \mathscr{E}$ above, direct computation yields
\begin{align}\label{w_energy}
\frac{d}{dt} \mathscr{E} =&  \int_{\R} e^{2c \lambda_k t} \langle k \rangle^{2m} \langle{k}^{-1} \rangle^{2\varepsilon}  \Big(\frac{d}{dt} E_{ k}[\omega_{k}] + 2c \lambda_k E_{ k}[\omega_{k}]\Big) \,{dk}= \mathcal{L}_{\omega} + \mathcal{N}_{\omega}  + 2c \lambda_k \mathscr{E},
\end{align}
where
\begin{align*}
\mathcal{L}_{\omega} &= 2  \int_{\R} e^{2c \lambda_k t} \langle k \rangle^{2m} \langle{k}^{-1} \rangle^{2\varepsilon}   \operatorname{Re} \langle \omega_k,  ( 1+ c_\tau \mathfrak{J}_k)\mathcal{L}_{\omega, k}\rangle \,{dk}\nn\\
&\quad + 2c_\alpha \int_{\R} e^{2c \lambda_k t}  \alpha(k) \langle k \rangle^{2m} \langle{k}^{-1} \rangle^{2\varepsilon}  \operatorname{Re} \langle \partial_{y} \omega_k, (1+ c_\tau\mathfrak{J}_k) \partial_{y} \mathcal{L}_{\omega, k}\rangle \,{dk}\nn\\
&\quad  -c_\beta   \int_{\R} \mathbf{1}_{\{|k| \ge \nu\}} e^{2c \lambda_k t} \beta(k) \langle k \rangle^{2m} \langle{k}^{-1} \rangle^{2\varepsilon}  \Big[ \operatorname{Re} \langle ik  \omega_k,  \partial_{y} \mathcal{L}_{\omega, k}\rangle  + \operatorname{Re} \langle ik \mathcal{L}_{\omega, k},  \partial_{y}\omega_k \rangle\Big]\,{dk},
\end{align*}
and
\begin{align*}
\mathcal{N}_{\omega} &= 2  \int_{\R} e^{2c \lambda_k t} \langle k \rangle^{2m} \langle{k}^{-1} \rangle^{2\varepsilon}   \operatorname{Re} \langle \omega_k,  ( 1+ c_\tau \mathfrak{J}_k) \mathcal{N}_{\omega, k} \rangle \,{dk}\nn\\
&\quad + 2c_\alpha \int_{\R} e^{2c \lambda_k t}  \alpha(k) \langle k \rangle^{2m} \langle{k}^{-1} \rangle^{2\varepsilon}  \operatorname{Re} \langle \partial_{y} \omega_k, (1+ c_\tau\mathfrak{J}_k) \partial_{y} \mathcal{N}_{\omega, k} \rangle \,{dk}\nn\\
&\quad  -c_\beta   \int_{\R} \mathbf{1}_{\{|k| \ge \nu\}} e^{2c \lambda_k t} \beta(k) \langle k \rangle^{2m} \langle{k}^{-1} \rangle^{2\varepsilon}  \Big[ \operatorname{Re} \langle ik  \omega_k,  \partial_{y} \mathcal{N}_{\omega, k} \rangle  + \operatorname{Re} \langle ik \mathcal{N}_{\omega, k},  \partial_{y}\omega_k \rangle\Big] \,{dk},\nn\\
& \stackrel{\mathrm{def}}{=} \mathcal{N}_{1} + \mathcal{N}_{2} + \mathcal{N}_{3}.
\end{align*}
For the linear term $ \mathcal{L}_{\omega}$, we can directly apply Proposition \ref{linear_es} to obtain
\begin{align}
	\mathcal{L}_{\omega} &\le -c_0 \mathscr{D} - c_0 \lambda_k \mathscr{E}.\label{lw=}
\end{align}
By choosing $ c \leq \frac{1}{4} c_0 $, and combining \eqref{w_energy} and \eqref{lw=}, we obtain
\begin{align}\label{energy}
	\frac{d}{dt} \mathscr{E} + 3c \mathscr{D} + 2c\lambda_k \mathscr{E} \le \mathcal{N}_{1} + \mathcal{N}_{2} + \mathcal{N}_{3}.
\end{align}
Next, we need only to control the nonlinear terms $(\mathcal{N}_{1} , \mathcal{N}_{2} , \mathcal{N}_{3} )$. Specifically, we establish the following three lemmas.
\begin{lemma}\label{le_1}
	Under the conditions of Theorem \ref{thm1.2}, there exists a constant $C$, depending only on $m$ and $ \varepsilon $, such that  $ \forall t>0$, there holds
	\begin{align}\label{N1}
		\mathcal{N}_{1} \le C\nu^{-\frac{1}{2}} \mathscr{E}^{\frac{1}{2}} \mathscr{D}_1^{\frac{1}{2}} \mathscr{D}_4^{\frac{1}{2}} + C\nu^{-\frac{1}{2}} \mathscr{E}^{\frac{1}{2}} \mathscr{D}_1^{\frac{1}{2}} \mathscr{D}_3^{\frac{1}{2}}+ C\nu^{-\frac{1}{2}} \mathscr{E}^{\frac{1}{2}} \mathscr{D}_1^{\frac{1}{4}} \mathscr{D}_3^{\frac{3}{4}}.
	\end{align}
	\begin{proof}
		First, we need to decompose the term $ \mathcal{N}_1$ into the following two components for estimation
		\begin{align}\label{N1_decom}
			\mathcal{N}_{1} &= 2 \int_{\R^2} e^{2c \lambda_k t} \langle k \rangle^{2m} \langle{k}^{-1} \rangle^{2\varepsilon}   \operatorname{Re} \langle \omega_k,  ( 1+ c_\tau \mathfrak{J}_k)  i\ell \phi_{\ell} \cdot \partial_{y} \omega_{k-\ell} \rangle    \,{dk} \,{d\ell}\nn\\
			& \quad -  2 \int_{\R^2} e^{2c \lambda_k t} \langle k \rangle^{2m} \langle{k}^{-1} \rangle^{2\varepsilon}   \operatorname{Re} \langle \omega_k,  ( 1+ c_\tau \mathfrak{J}_k)  \partial_{y} \phi_{\ell} \cdot i(k-\ell) \omega_{k-\ell} \rangle    \,{dk} \,{d\ell}\nn\\
			&\stackrel{\mathrm{def}}{=} \mathcal{N}_1^x + \mathcal{N}_1^y.
		\end{align}
		For the term $ \mathcal{N}_1^x$, we perform a detailed partitioning of the frequency space. When $ \frac{|k-\ell|}{2} \le |k| \le 2|k-\ell|$, applying Young's inequality, Lemmas \ref{le2.1} and \ref{pre1}, we have
		\begin{align}\label{n1x_klneq}
			 \left| \mathcal{N}_1^x \right| &= \left| 2 \int_{\R^2} \mathbf{1}_{\frac{|k-\ell|}{2} \le |k| \le 2|k-\ell|} e^{2c \lambda_k t} \langle k \rangle^{2m} \langle{k}^{-1} \rangle^{2\varepsilon}   \operatorname{Re} \langle \omega_k,  ( 1+ c_\tau \mathfrak{J}_k)  i\ell \phi_{\ell} \cdot \partial_{y} \omega_{k-\ell} \rangle    \,{dk} \,{d\ell} \right|\nn\\
			 & \lesssim \left\| e^{c \lambda_k t} \langle k \rangle^{m} \langle{k}^{-1} \rangle^{\varepsilon} \| \omega_k \|_{L^2_y}  \right\|_{L^2_k} \left\| |\ell| \| \phi_{\ell}\|_{L^\infty_y }\right\|_{L^1_\ell} \left\| e^{c \lambda_k t} \langle k \rangle^{m} \langle{k}^{-1} \rangle^{\varepsilon} \| \partial_{y} \omega_k \|_{L^2_y}  \right\|_{L^2_k}  \nn\\
			 &\lesssim \nu^{-\frac{1}{2}} \mathscr{E}^{\frac{1}{2}} \mathscr{D}_1^{\frac{1}{2}} \mathscr{D}_4^{\frac{1}{2}}.
		\end{align}
 When $2|k-\ell| \le |k| $, with $ |k| \approx |\ell|$ and $ |k-\ell|^{\frac{1}{2}} \lesssim |\ell|^{\frac{1}{2}}$, we similarly apply Lemmas \ref{le2.1} and \ref{pre1} to obtain
		\begin{align}\label{n1x_2}
			\left| \mathcal{N}_1^x \right| &= \left| 2 \int_{\R^2} \mathbf{1}_{2|k-\ell| \le |k|  } e^{2c \lambda_k t} \langle k \rangle^{2m} \langle{k}^{-1} \rangle^{2\varepsilon}   \operatorname{Re} \langle \omega_k,  ( 1+ c_\tau \mathfrak{J}_k)  i\ell \phi_{\ell} \cdot \partial_{y} \omega_{k-\ell} \rangle    \,{dk} \,{d\ell} \right|\nn\\
			& \lesssim \left\| e^{c \lambda_k t} \langle k \rangle^{m} \langle{k}^{-1} \rangle^{\varepsilon} \| \omega_k \|_{L^2_y}  \right\|_{L^2_k} \left\| |k|^{-\frac{1}{2}}\| \partial_{y} \omega_k \|_{L^2_y }\right\|_{L^1_k} \left\| e^{c \lambda_\ell t} \langle \ell \rangle^{m} \langle{\ell}^{-1} \rangle^{\varepsilon} |\ell|^{\frac{3}{2}} \| \phi_{\ell}  \|_{L^\infty_y}  \right\|_{L^2_\ell}  \nn\\
			&\lesssim \nu^{-\frac{1}{2}} \mathscr{E}^{\frac{1}{2}} \mathscr{D}_1^{\frac{1}{2}} \mathscr{D}_4^{\frac{1}{2}}.
		\end{align}
		When $ 2|k| \le |k-\ell|$, we have $ |\ell| \approx |k-\ell|$ and $ \lambda_k \lesssim \lambda_\ell$.  For $ \varepsilon \in (0, \frac{1}{12}) $, it follows that $ |\ell|^{-\frac{1}{2}} \lesssim \langle{\ell}^{-1} \rangle^{\varepsilon} \langle\frac{1}{k-\ell} \rangle^{\varepsilon} \langle{k}^{-1} \rangle^{\frac{1}{2} - 2\varepsilon}$. Furthermore, combining this with the fact that $ \langle k \rangle^{2m} \lesssim  \langle \ell \rangle^m \langle k-\ell \rangle^m$ and Lemma \ref{pre1}, we obtain
		\begin{align}\label{l=k-l}
		\left| \mathcal{N}_1^x \right| &= \left| 2 \int_{\R^2} \mathbf{1}_{2|k| \le |k-\ell|  } e^{2c \lambda_k t} \langle k \rangle^{2m} \langle{k}^{-1} \rangle^{2\varepsilon}   \operatorname{Re} \langle \omega_k,  ( 1+ c_\tau \mathfrak{J}_k)  i\ell\phi_{\ell} \cdot \partial_{y} \omega_{k-\ell} \rangle    \,{dk} \,{d\ell} \right|\nn\\
		& \lesssim \left\|   \langle{k}^{-1} \rangle^{\frac{1}{2}} \| \omega_k \|_{L^2_y}  \right\|_{L^1_k} \left\| e^{c \lambda_k t} \langle k \rangle^{m} \langle{k}^{-1} \rangle^{\varepsilon} \| \partial_{y} \omega_k \|_{L^2_y }\right\|_{L^2_k} \left\| e^{c \lambda_\ell t} \langle \ell \rangle^{m} \langle{\ell}^{-1} \rangle^{\varepsilon} |\ell|^{\frac{3}{2}} \| \phi_{\ell}  \|_{L^\infty_y}  \right\|_{L^2_\ell}  \nn\\
		&\lesssim \nu^{-\frac{1}{2}} \mathscr{E}^{\frac{1}{2}} \mathscr{D}_1^{\frac{1}{2}} \mathscr{D}_4^{\frac{1}{2}}.
		\end{align}
Next, we analyze the term $ \mathcal{N}_1^y $. We begin by decomposing the frequency space as follows:
		\begin{align}\label{1split}
			 \mathbf{1} = \mathbf{1}_{\frac{|k-\ell|}{2} \le |k| \le 2|k-\ell|} \left( \mathbf{1}_{|k-\ell|\ge \nu}  + \mathbf{1}_{|k-\ell|\le \nu} \right) + \mathbf{1}_{ 2|k-\ell| \le |k|} + \mathbf{1}_{ 2|k| \le |k-\ell|}.
		\end{align}	
We then estimate each term separately. For the first term in \eqref{1split}, since $ |k| \approx |k-\ell| $ and $ k \gtrsim \nu $, Young's inequality and Lemma \ref{pre1} imply
		\begin{align*}
		\left| \mathcal{N}_1^y \right| &= \left| 2 \int_{\R^2}\mathbf{1}_{\frac{\nu}{2}\le\frac{|k-\ell|}{2} \le |k| \le 2|k-\ell|}  e^{2c \lambda_k t} \langle k \rangle^{2m} \langle{k}^{-1} \rangle^{2\varepsilon}   \operatorname{Re} \langle \omega_k,  ( 1+ c_\tau \mathfrak{J}_k)  \partial_{y} \phi_{\ell} \cdot i(k-\ell) \omega_{k-\ell} \rangle    \,{dk} \,{d\ell} \right|\nn\\
		& \lesssim \left\| e^{c \lambda_k t} \langle k \rangle^{m} \langle{k}^{-1} \rangle^{\varepsilon} |k|^{\frac{1}{3}}\| \omega_k \|_{L^2_y}  \right\|_{L^2_{\{|k| \gtrsim \nu\}}} \left\|  \| \partial_{y} \phi_{\ell}\|_{L^\infty_y }\right\|_{L^1_\ell} \left\| e^{c \lambda_k t} \langle k \rangle^{m} \langle{k}^{-1} \rangle^{\varepsilon} |k|^{\frac{2}{3}}\|  \omega_k \|_{L^2_y}  \right\|_{L^2_{\{|k| \gtrsim \nu\}}}  \nn\\
		&\lesssim  \left(\left\| e^{c \lambda_k t} \langle k \rangle^{m} \langle{k}^{-1} \rangle^{\varepsilon} |k|^{\frac{1}{3}}\| \omega_k \|_{L^2_y}  \right\|_{L^2_{\{|k| \gtrsim \nu\}}} \right)^{\frac{3}{2}}\left(\left\| e^{c \lambda_k t} \langle k \rangle^{m} \langle{k}^{-1} \rangle^{\varepsilon} |k| \|  \omega_k \|_{L^2_y}  \right\|_{L^2_{\{|k| \gtrsim \nu\}}}   \right)^{\frac{1}{2}}\nn\\
&\qquad\times          \left\|  \| \partial_{y} \phi_{\ell}\|_{L^\infty_y }\right\|_{L^1_\ell} \nn\\
		&\lesssim \nu^{-\frac{1}{2}} \mathscr{E}^{\frac{1}{2}} \mathscr{D}_1^{\frac{1}{4}} \mathscr{D}_3^{\frac{3}{4}}.
		\end{align*}
		For the second term in \eqref{1split}, we have  $ \alpha(k) \approx 1$. Using a similar estimate as above, we obtain
		\begin{align*}
			\left| \mathcal{N}_1^y \right| &= \left| 2 \int_{\R^2}\mathbf{1}_{\frac{|k-\ell|}{2} \le |k| \le 2|k-\ell|\le2\nu}   e^{2c \lambda_k t} \langle k \rangle^{2m} \langle{k}^{-1} \rangle^{2\varepsilon}   \operatorname{Re} \langle \omega_k,  ( 1+ c_\tau \mathfrak{J}_k)  \partial_{y} \phi_{\ell} \cdot i(k-\ell) \omega_{k-\ell} \rangle    \,{dk} \,{d\ell} \right|\nn\\
			& \lesssim \nu^{\frac{1}{2}} \left\| e^{c \lambda_k t} \langle k \rangle^{m} \langle{k}^{-1} \rangle^{\varepsilon} \| \omega_k \|_{L^2_y}  \right\|_{L^2_{\{|k| \lesssim \nu\}}} \left\|  \| \partial_{y} \phi_{\ell}\|_{L^\infty_y }\right\|_{L^1_\ell} \left\| e^{c \lambda_k t} \langle k \rangle^{m} \langle{k}^{-1} \rangle^{\varepsilon} |k|^{\frac{1}{2}}\|  \omega_k \|_{L^2_y}  \right\|_{L^2_{\{|k| \lesssim \nu\}}}  \nn\\
			&\lesssim \nu^{-\frac{1}{2}} \left(\nu^{\frac{1}{2}}\left\| e^{c \lambda_k t} \langle k \rangle^{m} \langle{k}^{-1} \rangle^{\varepsilon} \| \omega_k \|_{L^2_y}  \right\|_{L^2_{\{|k| \lesssim \nu\}}}\right)^{\frac{3}{2}} \left\|  \| \partial_{y} \phi_{\ell}\|_{L^\infty_y }\right\|_{L^1_\ell} \nn\\
&\qquad\times\left( \nu^{\frac{1}{2}}\left\| e^{c \lambda_k t} \langle k \rangle^{m} \langle{k}^{-1} \rangle^{\varepsilon} |k|\|  \omega_k \|_{L^2_y}  \right\|_{L^2_{\{|k| \lesssim \nu\}}} \right)^{\frac{1}{2}} \nn\\
			&\lesssim \nu^{-\frac{1}{2}} \mathscr{E}^{\frac{1}{2}} \mathscr{D}_1^{\frac{1}{4}} \mathscr{D}_3^{\frac{3}{4}}.
		\end{align*}
Now consider the third term in \eqref{1split}. The condition $ 2|k-\ell| \le |k|$ implies $ |k| \approx |\ell|$, and hence by Lemma \ref{pre1},
		\begin{align*}
			\left| \mathcal{N}_1^y \right| &= \left| 2 \int_{\R^2} \mathbf{1}_{2|k-\ell| \le |k| }   e^{2c \lambda_k t} \langle k \rangle^{2m} \langle{k}^{-1} \rangle^{2\varepsilon}   \operatorname{Re} \langle \omega_k,  ( 1+ c_\tau \mathfrak{J}_k)  \partial_{y} \phi_{\ell} \cdot i(k-\ell) \omega_{k-\ell} \rangle    \,{dk} \,{d\ell} \right|\nn\\
			&\lesssim \left\| e^{c \lambda_k t} \langle k \rangle^{m} \langle{k}^{-1} \rangle^{\varepsilon} \| \omega_k \|_{L^2_y}  \right\|_{L^2_k}  \left\| e^{c \lambda_\ell t} \langle \ell \rangle^{m} \langle{\ell}^{-1} \rangle^{\varepsilon} |\ell| \| \partial_{y} \phi_{\ell}  \|_{L^2_y}  \right\|_{L^2_\ell} \left\| \|  \omega_k \|_{L^\infty_y }\right\|_{L^1_k} \nn\\
			&\lesssim \nu^{-\frac{1}{2}} \mathscr{E}^{\frac{1}{2}} \mathscr{D}_1^{\frac{1}{2}} \mathscr{D}_4^{\frac{1}{2}}.
		\end{align*}
		Finally, for the last term in \eqref{1split}, we proceed analogously to the treatment of \eqref{l=k-l}. Applying Lemma \ref{pre1} yields
		\begin{align*}
			\left| \mathcal{N}_1^y \right| &= \left| 2 \int_{\R^2} \mathbf{1}_{2|k| \le |k-\ell| }   e^{2c \lambda_k t} \langle k \rangle^{2m} \langle{k}^{-1} \rangle^{2\varepsilon}   \operatorname{Re} \langle \omega_k,  ( 1+ c_\tau \mathfrak{J}_k)  \partial_{y} \phi_{\ell} \cdot i(k-\ell) \omega_{k-\ell} \rangle    \,{dk} \,{d\ell} \right|\nn\\
			&\lesssim \left\|   \langle{k}^{-1} \rangle^{\frac{1}{2}} \| \omega_k \|_{L^2_y}  \right\|_{L^1_k}  \left\| e^{c \lambda_k t} \langle k \rangle^{m} \langle{k}^{-1} \rangle^{\varepsilon} |k|^{\frac{1}{2}} \| \omega_k \|_{L^\infty_y }\right\|_{L^2_k}  \left\| e^{c \lambda_\ell t} \langle \ell \rangle^{m} \langle{\ell}^{-1} \rangle^{\varepsilon} |\ell| \| \partial_{y} \phi_{\ell}  \|_{L^2_y}  \right\|_{L^2_\ell}  \nn\\
			&\lesssim \nu^{-\frac{1}{2}} \mathscr{E}^{\frac{1}{2}} \mathscr{D}_1^{\frac{1}{2}} \mathscr{D}_4^{\frac{1}{2}}.
		\end{align*}
		Combining the estimates for $  \mathcal{N}_1^x$ and $  \mathcal{N}_1^y $ above, we directly obtain \eqref{N1}, thereby completing the proof of this lemma.
	\end{proof}
\end{lemma}

\begin{lemma}\label{le_2}
	Under the conditions of Theorem \ref{thm1.2}, there exists a constant $C$, depending only on $m$ and $ \varepsilon $, such that  $ \forall t>0$, there holds
	\begin{align}\label{N2}
	\mathcal{N}_{2} \le C\nu^{-\frac{1}{2}}    \mathscr{E}^{\frac{1}{2} }  \mathscr{D}_2^{\frac{1}{2}}  \mathscr{D}_4^{\frac{1}{2}} + C\nu^{-\frac{1}{2}} \mathscr{E}^{\frac{1}{2} } \mathscr{D}_2^{\frac{1}{2}} \mathscr{D}_1^{\frac{1}{4}} \mathscr{D}_3^{\frac{1}{4}}+ C\nu^{-\frac{1}{2}} \mathscr{E}^{\frac{1}{2} } \mathscr{D}_2^{\frac{1}{2}} \mathscr{D}_5^{\frac{1}{4}}.
	\end{align}
	\begin{proof}
		First, analogous to the decomposition in \eqref{N1_decom}, we express the velocity in terms of the stream function, yielding the following decomposition
		\begin{align*}
		\mathcal{N}_{2} &= -2c_{\alpha} \int_{\R^2} e^{2c \lambda_k t} \alpha(k) \langle k \rangle^{2m} \langle{k}^{-1} \rangle^{2\varepsilon}   \operatorname{Re} \langle \partial_{y} ( 1+ c_\tau \mathfrak{J}_k) \partial_{y} \omega_k,    \left(i\ell\phi_{\ell} \cdot \partial_{y} \omega_{k-\ell}\right) \rangle    \,{dk} \,{d\ell}\nn\\
		& \quad +  2 c_{\alpha} \int_{\R^2} e^{2c \lambda_k t} \alpha(k) \langle k \rangle^{2m} \langle{k}^{-1} \rangle^{2\varepsilon}   \operatorname{Re} \langle \partial_{y} ( 1+ c_\tau \mathfrak{J}_k) \partial_{y} \omega_k,  \left(  \partial_{y} \phi_{\ell} \cdot i(k-\ell) \omega_{k-\ell} \right) \rangle    \,{dk} \,{d\ell}\nn\\
		&\stackrel{\mathrm{def}}{=} \mathcal{N}_2^x + \mathcal{N}_2^y.
		\end{align*}
		We first estimate the $ \mathcal{N}_2^x$ term in the above expression by performing the following decomposition of the frequency space
		\begin{align}\label{2_split}
			\mathbf{1} = \mathbf{1}_{\frac{|k-\ell|}{2} \le |k| \le 2|k-\ell|}  + \mathbf{1}_{ 2|k-\ell| \le |k|} \cdot \left( \mathbf{1}_{|k-\ell|\ge \nu}  + \mathbf{1}_{|k-\ell|\le \nu} \mathbf{1}_{|k |\ge \nu} + \mathbf{1}_{|k-\ell|\le \nu} \mathbf{1}_{|k |\le \nu}\right) + \mathbf{1}_{ 2|k| \le |k-\ell|}.
		\end{align}
		For the first term in the above expression, it follows directly from Young's inequality and Lemma \ref{pre1} that
		\begin{align*}
			\left|\mathcal{N}^x_{2} \right| &= \left| 2c_{\alpha}\! \! \int_{\R^2}\mathbf{1}_{\frac{|k-\ell|}{2} \le |k| \le 2|k-\ell|} e^{2c \lambda_k t} \alpha(k) \langle k \rangle^{2m} \langle{k}^{-1} \rangle^{2\varepsilon}   \operatorname{Re} \langle \partial_{y} ( 1+ c_\tau \mathfrak{J}_k) \partial_{y} \omega_k,    \left(i\ell\phi_{\ell} \cdot \partial_{y} \omega_{k-\ell}\right) \rangle    \,{dk} \,{d\ell}\right|\nn\\
			&\lesssim \left\|  e^{c \lambda_k t} \sqrt{\alpha(k)} \langle k \rangle^{m} \langle{k}^{-1} \rangle^{\varepsilon} \| \nabla_k \partial_{y} \omega_k\|_{L^2_y} \right\|_{L^2_k} \left\|  e^{c \lambda_k t} \sqrt{\alpha(k)} \langle k \rangle^{m} \langle{k}^{-1} \rangle^{\varepsilon} \|  \partial_{y} \omega_k\|_{L^2_y} \right\|_{L^2_k}\nn\\
  &\qquad\times\left\| |\ell| \| \phi_\ell\|_{L^\infty_y} \right\|_{L^1_\ell}\nn\\
			&\lesssim \nu^{-\frac{1}{2}} \mathscr{E}^{\frac{1}{2} } \mathscr{D}_2^{\frac{1}{2}} \mathscr{D}_4^{\frac{1}{2}}.
		\end{align*}
		For the second term above, where $ \alpha(k) =  \nu^{\frac{2}{3}} |k|^{-\frac{2}{3}} \lesssim \nu^{\frac{1}{3}} |k|^{-\frac{1}{3}} \cdot \nu^{\frac{1}{3}} |k-\ell|^{-\frac{1}{3}}$, $ |k| \approx |\ell|$, and $ |\ell|^{-\frac{1}{2}} \lesssim |k-\ell|^{-\frac{1}{2}}$, we immediately obtain from Lemma \ref{pre1}
			\begin{align*}
		\left|\mathcal{N}^x_{2} \right| &= \left| 2c_{\alpha} \int_{\R^2}\mathbf{1}_{2\nu\le2|k-\ell| \le |k| }   e^{2c \lambda_k t} \alpha(k) \langle k \rangle^{2m} \langle{k}^{-1} \rangle^{2\varepsilon}   \operatorname{Re} \langle \partial_{y} ( 1+ c_\tau \mathfrak{J}_k) \partial_{y} \omega_k,    \left(i\ell\phi_{\ell} \cdot \partial_{y} \omega_{k-\ell}\right) \rangle    \,{dk} \,{d\ell}\right|\nn\\
		&\lesssim  \int_{2\nu\le2|k-\ell| \le |k| }   e^{2c \lambda_k t}   \langle k \rangle^{2m} \langle{k}^{-1} \rangle^{2\varepsilon}   \nu^{\frac{2}{3}} |k|^{-\frac{2}{3}} \| \nabla_k \partial_{y} \omega_k \|_{L^2} |\ell|^{\frac{3}{2}} \|\phi_\ell \|_{L^\infty}  |k-\ell|^{-\frac{1}{2}} \| \partial_{y} \omega_{k-\ell}\|_{L^2} \,{dk} \,{d\ell}\nn\\
		&\lesssim \left\|  e^{c \lambda_k t}  \nu^{\frac{1}{3}} |k|^{-\frac{1}{3}}   \langle k \rangle^{m} \langle{k}^{-1} \rangle^{\varepsilon} \| \nabla_k \partial_{y} \omega_k\|_{L^2_y} \right\|_{L^2_{\{|k| \gtrsim \nu\}}} \left\|  \nu^{\frac{1}{3}} |k-\ell|^{-\frac{1}{3}} |k-\ell|^{-\frac{1}{2}}  \|  \partial_{y} \omega_{k-\ell}\|_{L^2_y} \right\|_{L^1_{\{|k-\ell| \gtrsim \nu\}}} \nn\\
		&\quad \quad \times \left\| e^{c \lambda_\ell t} \langle \ell \rangle^{m} \langle{\ell}^{-1} \rangle^{\varepsilon} |\ell|^{\frac{3}{2}} \| \phi_\ell\|_{L^\infty_y} \right\|_{L^2_\ell} \nn\\
		&\lesssim \nu^{-\frac{1}{2}} \mathscr{E}^{\frac{1}{2} } \mathscr{D}_2^{\frac{1}{2}} \mathscr{D}_4^{\frac{1}{2}},
		\end{align*}
	 where we have utilized the inequality $ \|\langle  k \rangle^{-m} \langle{k}^{-1} \rangle^{- \varepsilon} |k|^{-\frac{1}{2}} \|_{L^2_k} \lesssim 1$.

Let $\Omega_1=\left\{(k,\ell)\in\R^2\  \big|\ 2|k-\ell| \le |k|,\quad|k-\ell|\le \nu,\quad |k| \ge \nu \right\}$.
	For the third term in \eqref{2_split}, where $ \alpha(k) =  \nu^{\frac{2}{3}} |k|^{-\frac{2}{3}}   \lesssim \nu^{\frac{1}{3}} |k|^{-\frac{1}{3}} \cdot 1$, we proceed similarly to the treatment of the previous expression to obtain with
		\begin{align*}
			\left|\mathcal{N}^x_{2} \right| &= \left| 2c_{\alpha} \int_{\Omega_1}   e^{2c \lambda_k t} \alpha(k) \langle k \rangle^{2m} \langle{k}^{-1} \rangle^{2\varepsilon}   \operatorname{Re} \langle \partial_{y} ( 1+ c_\tau \mathfrak{J}_k) \partial_{y} \omega_k,    \left(i\ell\phi_{\ell} \cdot \partial_{y} \omega_{k-\ell}\right) \rangle    \,{dk} \,{d\ell}\right|\nn\\
			&\lesssim \int_{\Omega_1}   e^{2c \lambda_k t}   \langle k \rangle^{2m} \langle{k}^{-1} \rangle^{2\varepsilon}   \nu^{\frac{1}{3}} |k|^{-\frac{1}{3}} \| \nabla_k \partial_{y} \omega_k \|_{L^2} |\ell|^{\frac{3}{2}} \|\phi_\ell \|_{L^\infty}  |k-\ell|^{-\frac{1}{2}} \| \partial_{y} \omega_{k-\ell}\|_{L^2} \,{dk} \,{d\ell}\nn\\
			&\lesssim \left\|  e^{c \lambda_k t}  \nu^{\frac{1}{3}} |k|^{-\frac{1}{3}}   \langle k \rangle^{m} \langle{k}^{-1} \rangle^{\varepsilon} \| \nabla_k \partial_{y} \omega_k\|_{L^2_y} \right\|_{L^2_{\{|k| \gtrsim \nu\}}} \left\|   |k-\ell|^{-\frac{1}{2}}  \|  \partial_{y} \omega_{k-\ell}\|_{L^2_y} \right\|_{L^1_{\{|k-\ell| \le \nu\}}} \nn\\
			&\quad \quad \times \left\| e^{c \lambda_\ell t} \langle \ell \rangle^{m} \langle{\ell}^{-1} \rangle^{\varepsilon} |\ell|^{\frac{3}{2}} \| \phi_\ell\|_{L^\infty_y} \right\|_{L^2_\ell} \nn\\
			&\lesssim \nu^{-\frac{1}{2}} \mathscr{E}^{\frac{1}{2} } \mathscr{D}_2^{\frac{1}{2}} \mathscr{D}_4^{\frac{1}{2}}.
		\end{align*}
		For the fourth term in \eqref{2_split}, where  $ \alpha(k) =1$, utilizing the equivalence of $|k|$ and $|\ell|$, it follows readily that$\Omega_1=\left\{(k,\ell)\in\R^2\  \big|\ 2|k-\ell| \le |k|\le\nu,|k-\ell|\le \nu,\quad |k| \ge \nu \right\}$.
		\begin{align*}
		\left|\mathcal{N}^x_{2} \right| &= \Big| 2c_{\alpha} \int_{\R^2}\mathbf{1}_{ 2|k-\ell| \le |k|\le\nu} \mathbf{1}_{|k-\ell|\le \nu}   e^{2c \lambda_k t} \alpha(k) \langle k \rangle^{2m} \langle{k}^{-1} \rangle^{2\varepsilon}  \nn\\
 &\qquad\qquad\times\operatorname{Re} \langle \partial_{y} ( 1+ c_\tau \mathfrak{J}_k) \partial_{y} \omega_k,    \left(i\ell\phi_{\ell} \cdot \partial_{y} \omega_{k-\ell}\right) \rangle    \,{dk} \,{d\ell}\Big|\nn\\
		&\lesssim \left\|  e^{c \lambda_k t}     \langle k \rangle^{m} \langle{k}^{-1} \rangle^{\varepsilon} \| \nabla_k \partial_{y} \omega_k\|_{L^2_y} \right\|_{L^2_{\{|k| \le \nu\}}} \left\|   |k-\ell|^{-\frac{1}{2}}  \|  \partial_{y} \omega_{k-\ell}\|_{L^2_y} \right\|_{L^1_{\{|k-\ell| \le \nu\}}} \nn\\
		&\quad \quad \times\left\| e^{c \lambda_\ell t} \langle \ell \rangle^{m} \langle{\ell}^{-1} \rangle^{\varepsilon} |\ell|^{\frac{3}{2}} \| \phi_\ell\|_{L^\infty_y} \right\|_{L^2_\ell} \nn\\
		&\lesssim \nu^{-\frac{1}{2}} \mathscr{E}^{\frac{1}{2} } \mathscr{D}_2^{\frac{1}{2}} \mathscr{D}_4^{\frac{1}{2}}.
		\end{align*}
		Next, we address the most challenging case where $  2|k| \le |k-\ell|$. We further subdivide this into three frequency intervals, specifically as follows
		\begin{align}\label{split3}
			\mathbf{1}_{ 2|k| \le |k-\ell|} = \mathbf{1}_{ 2|k| \le |k-\ell|} \mathbf{1}_{ |k| \ge \nu} + \mathbf{1}_{ 2|k| \le |k-\ell|} \mathbf{1}_{ |k| \le \nu} \left( \mathbf{1}_{ |k-\ell| \le \nu}  + \mathbf{1}_{ |k-\ell| \ge \nu} \right).
		\end{align}
		For the first term in \eqref{split3}, in the case where $ |\ell| \approx |k-\ell|$ and $ |\ell|^{-\frac{1}{2}} \lesssim |k-\ell|^{-\frac{1}{3}} |\ell|^{-\frac{1}{6}} \lesssim   |k-\ell|^{-\frac{1}{3}} \langle{\ell}^{-1} \rangle^{\varepsilon} \langle\frac{1}{k-\ell} \rangle^{\varepsilon} \langle{k}^{-1} \rangle^{\frac{1}{6} - 2\varepsilon}$, it follows from Lemma \ref{pre1} that
		\begin{align*}
			\left|\mathcal{N}^x_{2} \right| &= \left| 2c_{\alpha} \int_{\R^2}\mathbf{1}_{2\nu\le2|k| \le |k-\ell| }   e^{2c \lambda_k t} \alpha(k) \langle k \rangle^{2m} \langle{k}^{-1} \rangle^{2\varepsilon}   \operatorname{Re} \langle \partial_{y} ( 1+ c_\tau \mathfrak{J}_k) \partial_{y} \omega_k,    \left(i\ell\phi_{\ell} \cdot \partial_{y} \omega_{k-\ell}\right) \rangle    \,{dk} \,{d\ell}\right|\nn\\
			&\lesssim   \int_{\R^2}\mathbf{1}_{2\nu\le2|k| \le |k-\ell| }      \langle{k}^{-1} \rangle^{\frac{1}{6}} \nu^{\frac{1}{3}} |k|^{-\frac{2}{3}} \| \nabla_k \partial_{y} \omega_k \|_{L^2} \cdot  e^{c \lambda_\ell t} \langle \ell \rangle^{m} \langle{\ell}^{-1} \rangle^{\varepsilon} |\ell|^{\frac{3}{2}} \|\phi_\ell \|_{L^\infty}\nn\\
			&\quad \quad   \cdot e^{c \lambda_{k-\ell} t}  \nu^{\frac{1}{3}} |k-\ell|^{-\frac{1}{3}} \langle k-\ell \rangle^{m} \langle \frac{1}{k-\ell} \rangle^{\varepsilon}  \| \partial_{y} \omega_k\|_{L^2} \,{dk} \,{d\ell}\nn\\
			&\lesssim \left\|  e^{c \lambda_k t}  \nu^{\frac{1}{3}} |k|^{-\frac{2}{3}}    \langle{k}^{-1} \rangle^{\frac{1}{6}} \| \nabla_k \partial_{y} \omega_k\|_{L^2_y} \right\|_{L^1_{\{|k| \ge \nu\}}} \left\|  \nu^{\frac{1}{3}} |k|^{-\frac{1}{3}} \langle k \rangle^{m} \langle{k}^{-1} \rangle^{\varepsilon}  \|  \partial_{y} \omega_k\|_{L^2_y} \right\|_{L^2_{\{|k| \ge \nu\}}} \nn\\
			&\quad \quad \cdot \left\| e^{c \lambda_\ell t} \langle \ell \rangle^{m} \langle{\ell}^{-1} \rangle^{\varepsilon} |\ell|^{\frac{3}{2}} \| \phi_\ell\|_{L^\infty_y} \right\|_{L^2_\ell} \nn\\
			&\lesssim \nu^{-\frac{1}{2}} \mathscr{E}^{\frac{1}{2} } \mathscr{D}_2^{\frac{1}{2}} \mathscr{D}_4^{\frac{1}{2}},
		\end{align*}
		where in the last line we have utilized the fact that  $ \left\| |k|^{-\frac{1}{3}}  \langle k \rangle^{-m} \langle{k}^{-1} \rangle^{\frac{1}{6}-\varepsilon}\right\|_{L^2_{\{|k| \ge \nu \}}}  \lesssim 1.$
		
Let $\Omega_2=\left\{(k,\ell)\in\R^2\  \big|\ 2|k| \le |k-\ell|,\quad |k|\le \nu,\quad |k-\ell| \le \nu\right\}.$
For the second term in \eqref{split3}, where $ \alpha(k) = 1$, using the fact that  	$$ |\ell|^{-\frac{1}{2}} \lesssim \langle{\ell}^{-1} \rangle^{\varepsilon} \langle\frac{1}{k-\ell} \rangle^{\varepsilon} \langle{k}^{-1} \rangle^{\frac{1}{2} - 2\varepsilon} \quad \text{and} \quad  \|   \langle k \rangle^{-m} \langle{k}^{-1} \rangle^{\frac{1}{2}-\varepsilon}\|_{L^2_{|k| \le \nu }}  \lesssim 1,$$  we have
	\begin{align*}
	\left|\mathcal{N}^x_{2} \right| &= \left| 2c_{\alpha} \int_{\Omega_2}   e^{2c \lambda_k t} \alpha(k) \langle k \rangle^{2m} \langle{k}^{-1} \rangle^{2\varepsilon}   \operatorname{Re} \langle \partial_{y} ( 1+ c_\tau \mathfrak{J}_k) \partial_{y} \omega_k,    \left(i\ell\phi_{\ell} \cdot \partial_{y} \omega_{k-\ell}\right) \rangle    \,{dk} \,{d\ell}\right|\nn\\
	&\lesssim   \int_{\Omega_2}        \langle{k}^{-1} \rangle^{\frac{1}{2}} \| \nabla_k \partial_{y} \omega_k \|_{L^2} \cdot  e^{c \lambda_\ell t} \langle \ell \rangle^{m} \langle{\ell}^{-1} \rangle^{\varepsilon} |\ell|^{\frac{3}{2}} \|\phi_\ell \|_{L^\infty}\nn\\
	&\quad \quad   \cdot e^{c \lambda_{k-\ell} t}   \langle k-\ell \rangle^{m} \langle \frac{1}{k-\ell} \rangle^{\varepsilon}  \| \partial_{y} \omega_k\|_{L^2} \,{dk} \,{d\ell}\nn\\
	&\lesssim \left\|  e^{c \lambda_k t}      \langle{k}^{-1} \rangle^{\frac{1}{2}} \| \nabla_k \partial_{y} \omega_k\|_{L^2_y} \right\|_{L^1_{\{|k| \le \nu\}}} \left\|   \langle k \rangle^{m} \langle{k}^{-1} \rangle^{\varepsilon}  \|  \partial_{y} \omega_k\|_{L^2_y} \right\|_{L^2_{\{|k| \le \nu\}}} \nn\\
 &\qquad\times\left\| e^{c \lambda_\ell t} \langle \ell \rangle^{m} \langle{\ell}^{-1} \rangle^{\varepsilon} |\ell|^{\frac{3}{2}} \| \phi_\ell\|_{L^\infty_y} \right\|_{L^2_\ell} \nn\\
	&\lesssim \nu^{-\frac{1}{2}} \mathscr{E}^{\frac{1}{2} } \mathscr{D}_2^{\frac{1}{2}} \mathscr{D}_4^{\frac{1}{2}}.
	\end{align*}
	Additionally, let $\Omega_3=\left\{(k,\ell)\in\R^2\  \big|\ 2|k| \le |k-\ell|,\quad|k|\le \nu,\quad |k-\ell| \ge \nu\right\}$,
 for the last term in \eqref{split3}, given the facts that $ 1 \lesssim \nu^{-\frac{1}{3}} |\ell|^{\frac{1}{3}} \cdot \nu^{\frac{1}{3}} |k-\ell|^{-\frac{1}{3}}$ and $  |\ell|^{-\frac{1}{6}} \lesssim \langle{\ell}^{-1} \rangle^{\varepsilon} \langle\frac{1}{k-\ell} \rangle^{\varepsilon} \langle{k}^{-1} \rangle^{\frac{1}{6} - 2\varepsilon}$, we apply Lemma \ref{pre1}  to obtain

	\begin{align}\label{N2x_last}
		\left|\mathcal{N}^x_{2} \right| &= \left| 2c_{\alpha} \int_{\Omega_3}   e^{2c \lambda_k t}  \langle k \rangle^{2m} \langle{k}^{-1} \rangle^{2\varepsilon}   \operatorname{Re} \langle \partial_{y} ( 1+ c_\tau \mathfrak{J}_k) \partial_{y} \omega_k,    \left(i\ell\phi_{\ell} \cdot \partial_{y} \omega_{k-\ell}\right) \rangle    \,{dk} \,{d\ell}\right|\nn\\
		&\lesssim   \int_{\Omega_3}        \langle{k}^{-1} \rangle^{\frac{1}{6}} \| \nabla_k \partial_{y} \omega_k \|_{L^2} \cdot  e^{c \lambda_\ell t} \langle \ell \rangle^{m} \langle{\ell}^{-1} \rangle^{\varepsilon} \nu^{-\frac{1}{3}} |\ell|^{\frac{1}{2}} |\ell| \|\phi_\ell \|_{L^\infty}\nn\\
		&\quad \quad   \cdot  e^{c \lambda_{k-\ell} t} \nu^{\frac{1}{3}} |k-\ell|^{-\frac{1}{3}}  \langle k-\ell \rangle^{m} \langle \frac{1}{k-\ell} \rangle^{\varepsilon}  \| \partial_{y} \omega_k\|_{L^2} \,{dk} \,{d\ell}\nn\\
		&\lesssim \nu^{-\frac{1}{3}}  \left\| \langle{k}^{-1} \rangle^{\frac{1}{6}} \| \nabla_k \partial_{y} \omega_k \|_{L^2}\right\|_{L^1_{\{|k| \le \nu\}}} \mathscr{E}^{\frac{1}{2} }  \mathscr{D}_4^{\frac{1}{2}} \nn\\
		&\lesssim \nu^{-\frac{1}{2}+\varepsilon}    \mathscr{E}^{\frac{1}{2} }  \mathscr{D}_2^{\frac{1}{2}}  \mathscr{D}_4^{\frac{1}{2}}
	\end{align}
	 where we have utilized the basic inequality
	\begin{align*}
		\left\| \langle{k}^{-1} \rangle^{\frac{1}{6}} \| \nabla_k \partial_{y} \omega_k \|_{L^2}\right\|_{L^1_{\{|k| \le \nu\}}} \lesssim \nu^{-\frac{1}{2}} \mathscr{D}_2^{\frac{1}{2}} \cdot \| \langle k \rangle^{-m}\|_{L^4_{\{|k| \le \nu\}}}  \| \langle{k}^{-1} \rangle^{\frac{1}{6} - \varepsilon} \|_{L^4_{\{|k| \le \nu\}}}  \lesssim \nu^{-\frac{1}{6} + \varepsilon} \mathscr{D}_2^{\frac{1}{2}}.
	\end{align*}
	Combining the estimates from the decompositions in \eqref{2_split} and \eqref{split3} above, we obtain
	\begin{align}\label{N2x_estimate}
		\left|\mathcal{N}^x_{2} \right| \lesssim \nu^{-\frac{1}{2}}    \mathscr{E}^{\frac{1}{2} }  \mathscr{D}_2^{\frac{1}{2}}  \mathscr{D}_4^{\frac{1}{2}}.
	\end{align}
	Next, we proceed to estimate the term $ \mathcal{N}^y_{2}$. We begin by decomposing the frequency space as follows
	\begin{align*}
		\mathbf{1} = \mathbf{1}_{\frac{|\ell|}{2} \le |k| \le 2|\ell|} \left( \mathbf{1}_{\{|\ell| \ge \nu\}}  + \mathbf{1}_{|\ell|\le \nu} \right) + \mathbf{1}_{ 2|\ell| \le |k|} + \mathbf{1}_{ 2|k| \le |\ell|}.
	\end{align*}
	For the first case above, let $\Omega_4=\left\{(k,\ell)\in\R^2\  \big|\ {|\ell|}/{2} \le |k| \le 2|\ell| ,\quad|\ell| \ge \nu\right\}$,
 based on $ |k| \approx |\ell|$, $ |k| \gtrsim \nu$, $ 1 = |\ell|^{\frac{3}{2}} |\ell|^{-\frac{3}{2}} \lesssim |\ell|^{\frac{3}{2}} |k-\ell|^{-\frac{3}{2}}$, it follows from Lemma \ref{pre1} that
	\begin{align*}
			\left|\mathcal{N}^y_{2} \right| &= \left| 2c_{\alpha} \int_{\Omega_4 }   e^{2c \lambda_k t} \alpha(k) \langle k \rangle^{2m} \langle{k}^{-1} \rangle^{2\varepsilon}   \operatorname{Re} \langle \partial_{y} ( 1+ c_\tau \mathfrak{J}_k) \partial_{y} \omega_k,    \left(\partial_{y}  \phi_{\ell} \cdot i(k-\ell)  \omega_{k-\ell}\right) \rangle    \,{dk} \,{d\ell}\right|\nn\\
			&\lesssim \int_{\Omega_4 }   e^{2c \lambda_k t}   \langle k \rangle^{2m} \langle{k}^{-1} \rangle^{2\varepsilon}   \nu^{\frac{2}{3}} |k|^{-\frac{2}{3}} \| \nabla_k \partial_{y} \omega_k \|_{L^2} |\ell|^{\frac{3}{2}} \|\partial_{y} \phi_\ell \|_{L^\infty}  |k-\ell|^{-\frac{1}{2}} \|  \omega_{k-\ell}\|_{L^2} \,{dk} \,{d\ell} \nn\\
			&\lesssim \left\|  e^{c \lambda_k t}  \nu^{\frac{1}{3}} |k|^{-\frac{1}{3}}   \langle k \rangle^{m} \langle{k}^{-1} \rangle^{\varepsilon} \| \nabla_k \partial_{y} \omega_k\|_{L^2_y} \right\|_{L^2_{\{|k| \gtrsim \nu\}}} \left\|   |k|^{-\frac{1}{2}}  \|   \omega_k\|_{L^2_y} \right\|_{L^1_{k}} \nn\\
			&\quad \quad \cdot \left\| e^{c \lambda_\ell t} \langle \ell \rangle^{m} \langle{\ell}^{-1} \rangle^{\varepsilon} \nu^{\frac{1}{3}} |\ell|^{-\frac{1}{3}} |\ell|^{\frac{3}{2}} \| \partial_{y} \phi_\ell\|_{L^\infty_y} \right\|_{L^2_{\{|\ell|\ge\nu\}}} \nn\\
			&\lesssim \nu^{-\frac{1}{2}} \mathscr{E}^{\frac{1}{2} } \mathscr{D}_2^{\frac{1}{2}} \mathscr{D}_5^{\frac{1}{2}},
	\end{align*}
	where we have utilized the fact that $ \| \partial_{y} \phi_\ell\|_{L^\infty_y} \lesssim |\ell|^{-\frac{1}{2}} \| \partial_{y} \nabla_\ell \phi_\ell\|_{L^2_y}.$
	For the second case, where $ \alpha(k)  \approx 1$, we can apply a method similar to the previous expression to obtain
	\begin{align*}
		\left|\mathcal{N}^y_{2} \right| &= \Big| 2c_{\alpha} \int_{\R^2}\mathbf{1}_{\frac{|\ell|}{2} \le |k| \le 2|\ell|\le2\nu}    e^{2c \lambda_k t} \alpha(k)  \langle k \rangle^{2m} \langle{k}^{-1} \rangle^{2\varepsilon} \\
  &\qquad\quad\cdot\operatorname{Re} \langle \partial_{y} ( 1+ c_\tau \mathfrak{J}_k) \partial_{y} \omega_k,    \left(\partial_{y}  \phi_{\ell} \cdot i(k-\ell)  \omega_{k-\ell}\right) \rangle    \,{dk} \,{d\ell}\Big|\nn\\
		&\lesssim \left\|  e^{c \lambda_k t}     \langle k \rangle^{m} \langle{k}^{-1} \rangle^{\varepsilon} \| \nabla_k \partial_{y} \omega_k\|_{L^2_y} \right\|_{L^2_{\{|k| \lesssim \nu\}}} \left\|   |k|^{-\frac{1}{2}}  \|   \omega_k\|_{L^2_y} \right\|_{L^1_k} \nn\\
&\qquad\times\left\| e^{c \lambda_\ell t} \langle \ell \rangle^{m} \langle{\ell}^{-1} \rangle^{\varepsilon}  |\ell|^{\frac{3}{2}} \| \partial_{y} \phi_\ell\|_{L^\infty_y} \right\|_{L^2_{\{|\ell|\le \nu\}}} \nn\\
		&\quad \quad \cdot  \nn\\
		&\lesssim \nu^{-\frac{1}{2}} \mathscr{E}^{\frac{1}{2} } \mathscr{D}_2^{\frac{1}{2}} \mathscr{D}_5^{\frac{1}{2}}.
	\end{align*}
	For the third case, where 	$ |k| \approx |k-\ell|$, applying Lemmas \ref{pre2} and \ref{pre1}, we have
	\begin{align*}
		\left|\mathcal{N}^y_{2} \right| &= \left| 2c_{\alpha}\!  \int_{\R^2}\!\!\mathbf{1}_{ 2|\ell|  \le |k| }   e^{2c \lambda_k t} \alpha(k) \langle k \rangle^{2m} \langle{k}^{-1} \rangle^{2\varepsilon}   \operatorname{Re} \langle \partial_{y} ( 1+ c_\tau \mathfrak{J}_k) \partial_{y} \omega_k,    \left(\partial_{y}  \phi_{\ell} \cdot i(k-\ell)  \omega_{k-\ell}\right) \rangle    \,{dk} \,{d\ell}\right|\nn\\
		&\lesssim \left\|  e^{c \lambda_k t}  \sqrt{\alpha(k)}   \langle k \rangle^{m} \langle{k}^{-1} \rangle^{\varepsilon} \| \nabla_k \partial_{y} \omega_k\|_{L^2_y} \right\|_{L^2_{k}} \left\|  e^{c \lambda_k t}  \sqrt{\alpha(k)}   \langle k \rangle^{m} \langle{k}^{-1} \rangle^{\varepsilon} |k|  \|   \omega_k\|_{L^2_y} \right\|_{L^2_k} \nn\\
&\qquad\times\left\|  \| \partial_{y} \phi_\ell\|_{L^\infty_y} \right\|_{L^1_{\ell}} \nn\\
		&\lesssim \nu^{-\frac{1}{2}} \mathscr{E}^{\frac{1}{2} } \mathscr{D}_2^{\frac{1}{2}} \mathscr{D}_1^{\frac{1}{4}} \mathscr{D}_3^{\frac{1}{4}}.
	\end{align*}
For the fourth case, which is the most complex, the interaction of frequencies makes controlling the nonlinear terms challenging, complicating the closure of our energy estimates. Therefore, we still need to perform a frequency decomposition, specifically as follows
	\begin{align}\label{split4}
		\mathbf{1}_{ 2|k| \le |\ell|} = \mathbf{1}_{ 2|k| \le |\ell|} \left( \mathbf{1}_{ |k| \ge \nu} +  \mathbf{1}_{ |k| \le \nu} \mathbf{1}_{ |\ell| \le \nu} + \mathbf{1}_{ |k| \le \nu} \mathbf{1}_{ |\ell| \ge \nu} \right).
	\end{align}
	For the first term in \eqref{split4}, given $ |\ell| \approx |k-\ell| $,
$$
 |k-\ell| \lesssim |\ell|^{\frac{3}{2}} |\ell|^{-\frac{1}{3}} |\ell|^{-\frac{1}{6}} \lesssim |\ell|^{\frac{3}{2}} |\ell|^{-\frac{1}{3}} \langle{\ell}^{-1} \rangle^{\varepsilon} \langle \frac{1}{k-\ell} \rangle^{\varepsilon} \langle{k}^{-1} \rangle^{\frac{1}{6} - 2\varepsilon}, $$ along with $ \langle k \rangle^{2m} \lesssim \langle \ell \rangle^{m} \langle k-\ell \rangle^{m} $,  we obtain
	\begin{align*}
		\left|\mathcal{N}^y_{2} \right| &= \Big| 2c_{\alpha} \int_{\R^2}\mathbf{1}_{ 2\nu\le2|k|  \le |\ell| }   e^{2c \lambda_k t} \alpha(k) \langle k \rangle^{2m} \langle{k}^{-1} \rangle^{2\varepsilon}\nn\\
&\qquad\qquad\cdot\operatorname{Re} \langle \partial_{y} ( 1+ c_\tau \mathfrak{J}_k) \partial_{y} \omega_k,    \left(\partial_{y}  \phi_{\ell} \cdot i(k-\ell)  \omega_{k-\ell}\right) \rangle    \,{dk} \,{d\ell}\Big|\nn\\
		&\lesssim \left\| e^{c \lambda_k t} \nu^{\frac{1}{3}} |k|^{-\frac{1}{3}}  \langle{k}^{-1} \rangle^{\frac{1}{6} } \| \nabla_k \partial_{y} \omega_{k}\|_{L^2_y} |k|^{-\frac{1}{3}}\right\|_{L^1_{\{|k| \ge \nu\}}} \left\| \langle k \rangle^{m}  \langle{k}^{-1} \rangle^{\varepsilon } \|\omega_{k} \|_{L^2_y}\right\|_{L^2_k} \nn\\
		&\quad \quad \cdot \left\| e^{c \lambda_\ell t} \nu^{\frac{1}{3}} |\ell|^{-\frac{1}{3}} \langle \ell \rangle^{m}  \langle{\ell}^{-1} \rangle^{\varepsilon } |\ell|^{\frac{3}{2}} \|\partial_{y} \phi_{\ell} \|_{L^\infty} \right\|_{L^2_{\{|\ell| \ge \nu\}}}\nn\\
		&\lesssim \mathscr{E}^{\frac{1}{2} }  \mathscr{D}_5^{\frac{1}{2}} \left\| e^{c \lambda_k t} \nu^{\frac{1}{3}} |k|^{-\frac{1}{3}} \langle k \rangle^{m}  \langle{k}^{-1} \rangle^{\varepsilon }  \| \nabla_k \partial_{y} \omega_{k}\|_{L^2_y} \right\|_{L^2_{\{|k| \ge \nu\}}} \left\| \langle{k}^{-1} \rangle^{\frac{1}{6} } |k|^{-\frac{1}{3}} \langle k \rangle^{-m}  \langle{k}^{-1} \rangle^{-\varepsilon } \right\|_{L^2_{\{|k| \ge \nu\}}}\nn\\
		&\lesssim \nu^{-\frac{1}{2}} \mathscr{E}^{\frac{1}{2} } \mathscr{D}_2^{\frac{1}{2}} \mathscr{D}_5^{\frac{1}{2}}.
	\end{align*}
	For the second term in \eqref{split4}, let $\Omega_5=\left\{(k,\ell)\in\R^2\  \big|\  2|k|  \le |\ell|,\quad |k| \le \nu,\quad|\ell| \le \nu \right\}$,
 given $ |k-\ell| \lesssim  |\ell|^{\frac{3}{2}} |\ell|^{-\frac{1}{2}}  \lesssim |\ell|^{\frac{3}{2}}  \langle{\ell}^{-1} \rangle^{\varepsilon} \langle \frac{1}{k-\ell} \rangle^{\varepsilon} \langle{k}^{-1} \rangle^{\frac{1}{2} - 2\varepsilon}$, we can apply a similar estimation as in the previous expression to obtain
	\begin{align*}
		\left|\mathcal{N}^y_{2} \right| &= \left| 2c_{\alpha} \int_{\Omega_5}  e^{2c \lambda_k t} \alpha(k) \langle k \rangle^{2m} \langle{k}^{-1} \rangle^{2\varepsilon}   \operatorname{Re} \langle \partial_{y} ( 1+ c_\tau \mathfrak{J}_k) \partial_{y} \omega_k,    \left(\partial_{y}  \phi_{\ell} \cdot i(k-\ell)  \omega_{k-\ell}\right) \rangle    \,{dk} \,{d\ell}\right|\nn\\
		&\lesssim \nu^{-\frac{1}{2}} \mathscr{E}^{\frac{1}{2} } \mathscr{D}_2^{\frac{1}{2}} \mathscr{D}_5^{\frac{1}{2}}.
	\end{align*}
	For the last term in \eqref{split4}, let $\Omega_6=\left\{(k,\ell)\in\R^2\  \big|\  2|k|  \le |\ell|, \quad|k| \le \nu,\quad|\ell| \ge \nu \right\}$,  given $ |k-\ell|  \lesssim  |\ell|^{\frac{7}{6}} |\ell|^{-\frac{1}{6}} \lesssim |\ell|^{\frac{7}{6}}  \langle{\ell}^{-1} \rangle^{\varepsilon} \langle \frac{1}{k-\ell} \rangle^{\varepsilon} \langle{k}^{-1} \rangle^{\frac{1}{6} - 2\varepsilon}$, and proceeding similarly to the estimation in \eqref{N2x_last} while combining with Lemma \ref{pre3}, one can get
	\begin{align*}
		\left|\mathcal{N}^y_{2} \right| &= \left| 2c_{\alpha} \int_{\Omega_6}   e^{2c \lambda_k t} \alpha(k) \langle k \rangle^{2m} \langle{k}^{-1} \rangle^{2\varepsilon}   \operatorname{Re} \langle \partial_{y} ( 1+ c_\tau \mathfrak{J}_k) \partial_{y} \omega_k,    \left(\partial_{y}  \phi_{\ell} \cdot i(k-\ell)  \omega_{k-\ell}\right) \rangle    \,{dk} \,{d\ell}\right|\nn\\
		&\lesssim  \left\| e^{c \lambda_k t}   \langle{k}^{-1} \rangle^{\frac{1}{6} } \| \nabla_k \partial_{y} \omega_{k}\|_{L^2_y} \right\|_{L^1_{\{|k| \le \nu\}}}   \left\| \langle k \rangle^{m}  \langle{k}^{-1} \rangle^{\varepsilon } \|\omega_{k} \|_{L^2_y}\right\|_{L^2_k}\nn\\
&\qquad\times\left\| e^{c \lambda_\ell t}  \langle \ell \rangle^{m}  \langle{\ell}^{-1} \rangle^{\varepsilon } |\ell|^{\frac{7}{6}} \|\partial_{y} \phi_{\ell} \|_{L^\infty} \right\|_{L^2_{\{|\ell| \ge \nu\}}}\nn\\
		&\lesssim \nu^{-\frac{1}{2} + \varepsilon} \mathscr{E}^{\frac{1}{2} } \mathscr{D}_2^{\frac{1}{2}} \mathscr{D}_5^{\frac{1}{2}}.
	\end{align*}
	Combining the estimates for $\mathcal{N}^y_{2}$ across the various frequency partitions, we obtain the estimate for $ \mathcal{N}^y_{2}$
	\begin{align}\label{N2y_estimate}
	\left|\mathcal{N}^y_{2} \right| \lesssim \nu^{-\frac{1}{2}}    \mathscr{E}^{\frac{1}{2} }  \mathscr{D}_2^{\frac{1}{2}}  \mathscr{D}_5^{\frac{1}{2}} + \nu^{-\frac{1}{2}} \mathscr{E}^{\frac{1}{2} } \mathscr{D}_2^{\frac{1}{2}} \mathscr{D}_1^{\frac{1}{4}} \mathscr{D}_3^{\frac{1}{4}}.
	\end{align}
	Finally, by combining \eqref{N2x_estimate} and \eqref{N2y_estimate}, we complete the proof of this lemma.
	\end{proof}
\end{lemma}

\begin{lemma}\label{le_3}
	Under the conditions of Theorem \ref{thm1.2}, there exists a constant $C$, depending only on $m$ and $ \varepsilon $, such that  $ \forall t>0$,  there holds
	\begin{align}\label{N3}
	\mathcal{N}_{3} \lesssim \nu^{-\frac{1}{2}} \mathscr{E}^{\frac{1}{2}} \Big( \mathscr{D}_1^{\frac{3}{8}} \mathscr{D}_3^{\frac{5}{8}} +   \mathscr{D}_1^{\frac{1}{2}} \mathscr{D}_3^{\frac{1}{2}} + \mathscr{D}_1^{\frac{1}{2}} \mathscr{D}_4^{\frac{1}{2}} + \mathscr{D}_1^{\frac{3}{4}} \mathscr{D}_3^{\frac{1}{4}} +  \mathscr{D}_1^{\frac{1}{8}} \mathscr{D}_2^{\frac{1}{4}} \mathscr{D}_3^{\frac{3}{8}} \mathscr{D}_5^{\frac{1}{4}} +  \mathscr{D}_1^{\frac{1}{4}} \mathscr{D}_2^{\frac{1}{4}} \mathscr{D}_3^{\frac{1}{2}}\Big).
	\end{align}
	\begin{proof}
		Next, we consider the estimation of $ \mathcal{N}_{3} $. It suffices to focus on the first term of $ \mathcal{N}_{31} $, as the second term can be handled similarly after performing integration by parts.
		\begin{align*}
			\mathcal{N}_{3,1} = -c_\beta   \int_{\R} \mathbf{1}_{\{|k| \ge \nu\}} e^{2c \lambda_k t} \beta(k) \langle k \rangle^{2m} \langle{k}^{-1} \rangle^{2\varepsilon}   \operatorname{Re} \langle ik  \omega_k,  \partial_{y} \mathcal{N}_{\omega, k} \rangle   \,{dk} \stackrel{\mathrm{def}}{=} \mathcal{N}^x_{3,1} + \mathcal{N}^y_{3,1}.
		\end{align*}
		For the term $ \mathcal{N}^x_{3,1}$, by applying integration by parts and H\"older's inequality, we immediately obtain
		\begin{align*}
			\left|\mathcal{N}^x_{3,1}\right| &= \left| c_{\beta} \int_{\R^2} \mathbf{1}_{\{|k| \ge \nu\}} e^{2c \lambda_k t} \beta(k) \langle k \rangle^{2m} \langle{k}^{-1} \rangle^{2\varepsilon}   \operatorname{Re} \langle ik  \partial_{y} \omega_k,   \left( i\ell\phi_{\ell} \cdot \partial_{y} \omega_{k-\ell}\right) \rangle   \,{dk} \,{d\ell}\right|\nn\\
			&\lesssim \left|  \int_{\R^2} \mathbf{1}_{\{|k| \ge \nu\}} e^{2c \lambda_k t} \nu^{\frac{1}{3}} |k|^{-\frac{1}{3}} \langle k \rangle^{2m} \langle{k}^{-1} \rangle^{2\varepsilon}      \| \partial_{y} \omega_k\|_{L^2_y} \|\ell\phi_{\ell}\|_{L^\infty_y} \| \partial_{y} \omega_{k-\ell} \|_{L^2_y} \,{dk} \,{d\ell}\right|.
		\end{align*}
		Next, we perform a frequency decomposition identical to that in $ \mathcal{N}^x_{3,1}$:
		\begin{align*}
			\mathbf{1} = \mathbf{1}_{\frac{|k-\ell|}{2} \le |k| \le 2|k-\ell|}  + \mathbf{1}_{ 2|k-\ell| \le |k|} + \mathbf{1}_{ 2|k| \le |k-\ell|}.
		\end{align*}
		For one of the cases, let $\Omega_7=\left\{(k,\ell)\in\R^2\  \big|\  {|k-\ell|}/{2} \le |k| \le 2|k-\ell|,\quad |k|\ge \nu\right\}$, it follows directly from Young's inequality and Lemma \ref{pre1} that		
\begin{align*}
			\left|\mathcal{N}^x_{3,1}\right| &\lesssim \left|  \int_{\Omega_7}  e^{2c \lambda_k t} \nu^{\frac{1}{3}} |k|^{-\frac{1}{3}} \langle k \rangle^{2m} \langle{k}^{-1} \rangle^{2\varepsilon}      \| \partial_{y} \omega_k\|_{L^2_y} \|\ell\phi_{\ell}\|_{L^\infty_y} \| \partial_{y} \omega_{k-\ell} \|_{L^2_y} \,{dk} \,{d\ell}\right|\nn\\
			&\lesssim  \left\| e^{c \lambda_k t} \nu^{\frac{1}{3}} |k|^{-\frac{1}{3}} \langle k \rangle^{m} \langle{k}^{-1} \rangle^{\varepsilon}      \| \partial_{y} \omega_k\|_{L^2_y} \right\|_{L^2_{\{|k| \ge \nu\}}} \nn\\
&\qquad\times\left\| \|\ell\phi_{\ell}\|_{L^\infty_y} \right\|_{L^1_\ell} \left\| \langle k \rangle^{m} \langle{k}^{-1} \rangle^{\varepsilon}      \| \partial_{y} \omega_k\|_{L^2_y}\right\|_{L^2_k}\nn\\
			&\lesssim \nu^{-\frac{1}{2} } \mathscr{E}^{\frac{1}{2} } \mathscr{D}_1^{\frac{1}{2}} \mathscr{D}_4^{\frac{1}{2}}.
		\end{align*}
		For the second and third terms, by applying estimates similar to those in  \eqref{n1x_2} and \eqref{l=k-l}, we obtain
		\begin{align*}
		\left|\mathcal{N}^x_{3,1}\right| &\lesssim \left|  \int_{\R^2} \mathbf{1}_{\big\{ 2|k-\ell| \le |k|\big\}\cap \big\{ |k|\ge \nu\big\}} e^{2c \lambda_k t} \sqrt{\alpha(k)} \langle k \rangle^{2m} \langle{k}^{-1} \rangle^{2\varepsilon}      \| \partial_{y} \omega_k\|_{L^2_y} \|\ell\phi_{\ell}\|_{L^\infty_y} \| \partial_{y} \omega_{k-\ell} \|_{L^2_y} \,{dk} \,{d\ell} \right|\nn\\
		&\lesssim \nu^{-\frac{1}{2} } \mathscr{E}^{\frac{1}{2} } \mathscr{D}_1^{\frac{1}{2}} \mathscr{D}_4^{\frac{1}{2}},
		\end{align*}
		and
		\begin{align*}
			\left|\mathcal{N}^x_{3,1}\right| &\lesssim \left|  \int_{\R^2}\mathbf{1}_{\big\{ 2|k| \le |k-\ell|\big\}\cap \big\{ |k|\ge \nu\big\}}  e^{2c \lambda_k t} \sqrt{\alpha(k)} \langle k \rangle^{2m} \langle{k}^{-1} \rangle^{2\varepsilon}      \| \partial_{y} \omega_k\|_{L^2_y} \|\ell\phi_{\ell}\|_{L^\infty_y} \| \partial_{y} \omega_{k-\ell} \|_{L^2_y} \,{dk} \,{d\ell} \right|\nn\\
			&\lesssim \nu^{-\frac{1}{2} } \mathscr{E}^{\frac{1}{2} } \mathscr{D}_1^{\frac{1}{2}} \mathscr{D}_4^{\frac{1}{2}}.
		\end{align*}
		Collecting the above estimates for	$ \mathcal{N}^x_{3,1}$	in all the different cases together, we arrive at
		\begin{align}\label{N31x}
			\mathcal{N}^x_{3,1} \lesssim \nu^{-\frac{1}{2} } \mathscr{E}^{\frac{1}{2} } \mathscr{D}_1^{\frac{1}{2}} \mathscr{D}_4^{\frac{1}{2}}.
		\end{align}
	 Next, we address the most challenging term $\mathcal{N}^y_{3,1} $ in this paper. We begin by applying the derivative $ \partial_y$ to $ \partial_{y} \phi_{\ell}$ and $ \omega_{k-\ell}$ separately, followed by a detailed estimation of each resulting term
	    \begin{align*}
	    		\left|\mathcal{N}^y_{3,1}\right| =& \left| c_{\beta} \int_{\R^2} \mathbf{1}_{\{|k| \ge \nu\}} e^{2c \lambda_k t} \beta(k) \langle k \rangle^{2m} \langle{k}^{-1} \rangle^{2\varepsilon}   \operatorname{Re} \langle ik   \omega_k,   \partial_{y}\left( \partial_{y} \phi_{\ell} \cdot i(k-\ell) \omega_{k-\ell}\right) \rangle   \,{dk} \,{d\ell}\right|\nn\\
	    		\le& \left| c_{\beta} \int_{\R^2} \mathbf{1}_{\{|k| \ge \nu\}} e^{2c \lambda_k t} \beta(k) \langle k \rangle^{2m} \langle{k}^{-1} \rangle^{2\varepsilon}   \operatorname{Re} \langle ik   \omega_k,   \left( \partial^2_{y} \phi_{\ell} \cdot i(k-\ell) \omega_{k-\ell}\right) \rangle   \,{dk} \,{d\ell} \right|\nn\\
	    		& +  \left| c_{\beta} \int_{\R^2} \mathbf{1}_{\{|k| \ge \nu\}} e^{2c \lambda_k t} \beta(k) \langle k \rangle^{2m} \langle{k}^{-1} \rangle^{2\varepsilon}   \operatorname{Re} \langle ik   \omega_k,   \left( \partial_{y} \phi_{\ell} \cdot i(k-\ell) \partial_{y}\omega_{k-\ell}\right) \rangle   \,{dk} \,{d\ell} \right|\nn\\
 \stackrel{\mathrm{def}}{=}& \mathcal{N}^{y,1}_{3,1} + \mathcal{N}^{y,2}_{3,1}.
	    \end{align*}
	    For the term $\mathcal{N}^{y,1}_{3,1} $, we need to perform the following frequency decomposition
	    \begin{align*}
	    \mathbf{1} = \mathbf{1}_{\frac{|k-\ell|}{2} \le |k| \le 2|k-\ell|}  + \mathbf{1}_{ 2|k-\ell| \le |k|} \left( \mathbf{1}_{ \nu \le |k| \le 1} + \mathbf{1}_{  |k| \ge 1} \right) + \mathbf{1}_{ 2|k| \le |k-\ell|}.
	    \end{align*}
	    For the first case above,
 by applying Young's inequality along with Lemma \ref{pre2} and \ref{pre3}, we obtain
	    \begin{align*}
	    	\left|\mathcal{N}^{y,1}_{3,1}\right| &= \left| c_{\beta} \int_{\Omega_7}  e^{2c \lambda_k t} \nu^{\frac{1}{3}} |k|^{-\frac{4}{3}}  \langle k \rangle^{2m} \langle{k}^{-1} \rangle^{2\varepsilon}   \operatorname{Re} \langle ik   \omega_k,   \left( \partial^2_{y} \phi_{\ell} \cdot i(k-\ell) \omega_{k-\ell}\right) \rangle   \,{dk} \,{d\ell} \right|\\
	    	&\lesssim \int_{\Omega_7} e^{2c \lambda_k t} \nu^{\frac{1}{6}} |k|^{\frac{1}{3}}  \langle k \rangle^{2m} \langle{k}^{-1} \rangle^{2\varepsilon}   \| \omega_k\|_{L^2_y}   \| \partial^2_{y} \phi_{\ell} \|_{L^\infty_y} \nu^{\frac{1}{6}} |k-\ell|^{\frac{1}{3}} \| \omega_{k-\ell}\|_{L^2_y}    \,{dk} \,{d\ell}\\
	    	&\lesssim \nu^{\frac{1}{6}} \left( \int_{\{|k| \ge \nu\}} e^{2c \lambda_k t} |k|^{\frac{2}{3}}  \langle k \rangle^{2m} \langle{k}^{-1} \rangle^{2\varepsilon} \| \omega_{k} \|^2_{L^2_y}  \,{dk} \right)^{\frac{1}{2}} \left\| \| \partial^2_{y} \phi_{\ell} \|_{L^\infty_y}\right\|_{L^1_\ell} \cdot \mathscr{D}_3^{\frac{1}{2}}\nn\\
	    	&\lesssim \nu^{-\frac{1}{4}} \mathscr{E}^{\frac{1}{2}} \mathscr{D}_1^{\frac{3}{8}} \mathscr{D}_3^{\frac{5}{8}}.
	    \end{align*}
	    For the second case, let $\Omega_8=\left\{(k,\ell)\in\R^2\  \big|\  |k| \ge 2|k-\ell|,\quad  1 \ge |k|\ge \nu\right\}$,
 based on $ |k| \approx |\ell|$ and $ \nu^{\frac{1}{3}} |k|^{\frac{1}{6}} \lesssim \nu^{\frac{1}{6}} |k|^{\frac{1}{3}}$, along with Lemma \ref{pre1}, one can get
	    \begin{align*}
	    \left|\mathcal{N}^{y,1}_{3,1}\right| &= \left| c_{\beta} \int_{\Omega_8}  e^{2c \lambda_k t} \nu^{\frac{1}{3}} |k|^{-\frac{4}{3}}  \langle k \rangle^{2m} \langle{k}^{-1} \rangle^{2\varepsilon}   \operatorname{Re} \langle ik   \omega_k,   \left( \partial^2_{y} \phi_{\ell} \cdot i(k-\ell) \omega_{k-\ell}\right) \rangle   \,{dk} \,{d\ell} \right|\\
	    &\lesssim \nu^{-\frac{1}{2}} \int_{\Omega_8}  e^{2c \lambda_k t} \nu^{\frac{1}{6}} |k|^{\frac{1}{3}}  \langle k \rangle^{2m} \langle{k}^{-1} \rangle^{2\varepsilon}   \| \omega_k\|_{L^2_y}   \nu^{\frac{1}{2}} |\ell|^{\frac{1}{2}}\| \partial^2_{y} \phi_{\ell} \|_{L^\infty_y} \| \omega_{k-\ell}\|_{L^2_y}    \,{dk} \,{d\ell}\\
	    &\lesssim \nu^{-\frac{1}{2}} \mathscr{D}_3^{\frac{1}{2}}  \nu^{\frac{1}{2}}\left\|  \|e^{c \lambda_\ell t} \langle \ell \rangle^{m} \langle{\ell}^{-1} \rangle^{\varepsilon} |\ell|^{\frac{1}{2}}  \partial^2_{y} \phi_{\ell} \|_{L^\infty_y}\right\|_{L^2_\ell} \cdot \left\| \| \omega_{k}\|_{L^2_y} \right\|_{L^1_k}\nn\\
	    &\lesssim \nu^{-\frac{1}{2}} \mathscr{E}^{\frac{1}{2}} \mathscr{D}_1^{\frac{1}{2}} \mathscr{D}_3^{\frac{1}{2}}.
	    \end{align*}
	     The last line relies on the fact that  $$ \| |\ell|^{\frac{1}{2}}  \partial^2_{y} \phi_{\ell} \|_{L^\infty_y} \lesssim  \|\nabla_\ell \partial^2_{y} \phi_{\ell} \|_{L^2_y} \lesssim \|\nabla_\ell \omega_{\ell} \|_{L^2_y} \quad   \text{and} \quad  \|\langle  k \rangle^{-m} \langle{k}^{-1} \rangle^{- \varepsilon}  \|_{L^2_k} \lesssim 1.$$
	    For the third case, where $ |k| \approx |\ell|$, let $\Omega_{9}=\left\{(k,\ell)\in\R^2\  \big|\   |k| \ge 2|k-\ell|,\quad|k|\ge 1\right\},$
 by directly applying Lemma \ref{pre2} and \ref{pre3}, it follows readily that
	    \begin{align*}
	    	\left|\mathcal{N}^{y,1}_{3,1}\right| &= \left| c_{\beta} \int_{\Omega_{9}}  e^{2c \lambda_k t} \nu^{\frac{1}{3}} |k|^{-\frac{4}{3}}  \langle k \rangle^{2m} \langle{k}^{-1} \rangle^{2\varepsilon}   \operatorname{Re} \langle ik   \omega_k,   \left( \partial^2_{y} \phi_{\ell} \cdot i(k-\ell) \omega_{k-\ell}\right) \rangle   \,{dk} \,{d\ell} \right|\\
	    	&\lesssim  \int_{\Omega_{9}}  e^{2c \lambda_k t} \nu^{\frac{1}{12}} |k|^{\frac{1}{6}}  \langle k \rangle^{2m} \langle{k}^{-1} \rangle^{2\varepsilon}   \| \omega_k\|_{L^2_y}   \nu^{\frac{1}{4}}\| \partial^2_{y} \phi_{\ell} \|_{L^\infty_y} |k-\ell| \| \omega_{k-\ell}\|_{L^2_y}    \,{dk} \,{d\ell}\\
	    	&\lesssim  \left\| \nu^{\frac{1}{12}} |k|^{\frac{1}{6}} e^{c \lambda_k t} \langle k \rangle^m \langle{k}^{-1} \rangle^{\varepsilon} \| \omega \|_{L^2_y}\right\|_{L^2_{\{|k| \ge 1\}}} \cdot \left\| |k-\ell| \| \omega_{k}\|_{L^2_y} \right\|_{L^1_k}\nn\\
&\qquad\times\nu^{\frac{1}{4}}\left\|  \|e^{c \lambda_\ell t} \langle \ell \rangle^{m} \langle{\ell}^{-1} \rangle^{\varepsilon}   \partial^2_{y} \phi_{\ell} \|_{L^\infty_y}\right\|_{L^2_\ell}\nn\\
	    	&\lesssim \nu^{-\frac{1}{2}} \mathscr{E}^{\frac{1}{2}} \mathscr{D}_1^{\frac{3}{4}} \mathscr{D}_3^{\frac{1}{4}}.
	    \end{align*}
	    Additionally, for the fourth case, where $ |\ell| \approx |k-\ell|$ and $ |\ell|^{-\frac{1}{6}} \lesssim \langle{\ell}^{-1} \rangle^{\varepsilon} \langle\frac{1}{k-\ell} \rangle^{\varepsilon} \langle{k}^{-1} \rangle^{\frac{1}{6} - 2\varepsilon}$, by directly applying Lemmas \ref{pre2} and \ref{pre3}, we obtain
	    \begin{align*}
	    \left|\mathcal{N}^{y,1}_{3,1}\right| &\lesssim  \left|  \int_{\R^2}\mathbf{1}_{ 2\nu\le2 |k| \le |k-\ell| }  e^{2c \lambda_k t} \nu^{\frac{1}{3}} |k|^{-\frac{1}{3}}  \langle k \rangle^{2m} \langle{k}^{-1} \rangle^{2\varepsilon}   \| \omega_k\|_{L^2_y} \| \partial^2_{y} \phi_{\ell}\|_{L^\infty_y} \|(k-\ell) \omega_{k-\ell}\|_{L^2_y}   \,{dk} \,{d\ell} \right|\\
	    &\lesssim \int_{\R^2}\mathbf{1}_{ 2\nu\le2 |k| \le |k-\ell| } e^{2c \lambda_k t}  |k|^{-\frac{1}{3}}  \langle{k}^{-1} \rangle^{\frac{1}{6}}    \| \omega_k\|_{L^2_y}   \nu^{\frac{1}{3}} |\ell|^{\frac{2}{3}} \langle \ell \rangle^{m} \langle{\ell}^{-1} \rangle^{\varepsilon}  \| \partial^2_{y} \phi_{\ell} \|_{L^\infty_y}\nn\\
 &\qquad\times\langle k-\ell \rangle^{m} \langle\frac{1}{k-\ell} \rangle^{\varepsilon}  |k-\ell|^{\frac{1}{2}} \| \omega_{k-\ell}\|_{L^2_y}    \,{dk} \,{d\ell}\\
	    &\lesssim \left\| |k|^{-\frac{1}{3}} \langle{k}^{-1} \rangle^{\frac{1}{6}} \|\omega_{k} \|_{L^2_y} \right\|_{L^1_{\{|k| \ge \nu\}}} \left\| e^{c \lambda_\ell t}  \nu^{\frac{1}{3}} |\ell|^{\frac{2}{3}} \langle \ell \rangle^{m} \langle{\ell}^{-1} \rangle^{\varepsilon}  \| \partial^2_{y} \phi_{\ell} \|_{L^\infty_y}\right\|_{L^2_{\{|\ell| \gtrsim \nu\}}}  \nn\\
 &\qquad\times\left\|e^{c \lambda_k t} \langle k \rangle^{m} \langle\frac{1}{k} \rangle^{\varepsilon}  |k|^{\frac{1}{2}} \| \omega_{k}\|_{L^2_y}    \right\|_{L^2_{\{|k|\ge \nu\}}}\nn\\
	    &\lesssim \nu^{-\frac{1}{2}} \mathscr{E}^{\frac{1}{2}} \mathscr{D}_1^{\frac{1}{8}} \mathscr{D}_2^{\frac{1}{4}} \mathscr{D}_3^{\frac{3}{8}} \mathscr{D}_5^{\frac{1}{4}}.
	    \end{align*}
Finally, we address the term $ \mathcal{N}^{y,2}_{3,1}$. Using the same frequency decomposition as for $ \mathcal{N}^{x}_{1}$,
	    \begin{align*}
	    \mathbf{1} = \mathbf{1}_{\frac{|k-\ell|}{2} \le |k| \le 2|k-\ell|}  + \mathbf{1}_{ 2|k-\ell| \le |k|}  + \mathbf{1}_{ 2|k| \le |k-\ell|},
	    \end{align*}
we proceed to estimate each case. For the first case,
 By Lemmas \ref{pre1} and \ref{pre2}, we obtain
\begin{align*}
	    	\mathcal{N}^{y,2}_{3,1} &= \left| c_{\beta} \int_{\Omega_{7}}  e^{2c \lambda_k t} \nu^{\frac{1}{3}} |k|^{-\frac{4}{3}} \langle k \rangle^{2m} \langle{k}^{-1} \rangle^{2\varepsilon}   \operatorname{Re} \langle ik   \omega_k,   \left( \partial_{y} \phi_{\ell} \cdot i(k-\ell) \partial_{y}\omega_{k-\ell}\right) \rangle   \,{dk} \,{d\ell} \right| \\
	    	&\lesssim  \left\| e^{c \lambda_k t} \nu^{\frac{1}{3}} |k|^{\frac{2}{3}} \langle k \rangle^{m} \langle{k}^{-1} \rangle^{\varepsilon}  \|\omega_{k} \|_{L^2_y} \right\|_{L^2_{\{|k|\ge \nu\}}} \left\| e^{c \lambda_k t} \langle k \rangle^{m} \langle{k}^{-1} \rangle^{\varepsilon}  \|\partial_{y} \omega_{k} \|_{L^2_y}\right\|_{L^2_{\{|k|\gtrsim \nu\}}} \left\| \|\partial_{y}  \phi_{\ell} \|_{L^\infty_y} \right\|_{L^1_\ell}\\
	    	&\lesssim \nu^{-\frac{1}{2}} \mathscr{E}^{\frac{1}{2}} \mathscr{D}_1^{\frac{3}{4}} \mathscr{D}_3^{\frac{1}{4}}.
	    \end{align*}
	    For the second case, let $$\Omega_{10}=\left\{(k,\ell)\in\R^2\  \big|\ 2|k-\ell| \le |k|,\quad |k|\ge \nu\right\}.$$
A similar argument yields
\begin{align*}
	    	\mathcal{N}^{y,2}_{3,1} &= \left| c_{\beta} \int_{\Omega_{10}}  e^{2c \lambda_k t} \nu^{\frac{1}{3}} |k|^{-\frac{4}{3}} \langle k \rangle^{2m} \langle{k}^{-1} \rangle^{2\varepsilon}   \operatorname{Re} \langle ik   \omega_k,   \left( \partial_{y} \phi_{\ell} \cdot i(k-\ell) \partial_{y}\omega_{k-\ell}\right) \rangle   \,{dk} \,{d\ell} \right| \\
	    	&\lesssim \left\| e^{c \lambda_k t} \nu^{\frac{1}{3}} |k|^{\frac{2}{3}} \langle k \rangle^{m} \langle{k}^{-1} \rangle^{\varepsilon}  \|\omega_{k} \|_{L^2_y} \right\|_{L^2_{\{|k|\ge \nu\}}} \left\| |k|^{-\frac{1}{2}} \|\partial_{y} \omega_{k} \|_{L^2_y}\right\|_{L^1_{k}} \\
&\qquad\times
\left\| e^{c \lambda_k t} \langle \ell \rangle^{m} \langle{\ell}^{-1} \rangle^{\varepsilon} |\ell|^{\frac{1}{2}} \|\partial_{y}  \phi_{\ell} \|_{L^\infty_y} \right\|_{L^2_\ell}\\
	    	&\lesssim \nu^{-\frac{1}{2}} \mathscr{E}^{\frac{1}{2}} \mathscr{D}_1^{\frac{3}{4}} \mathscr{D}_3^{\frac{1}{4}}.
	    \end{align*}
	    For the third case, where $ |\ell|^{-\frac{1}{6}} \lesssim \langle{\ell}^{-1} \rangle^{\varepsilon} \langle\frac{1}{k-\ell} \rangle^{\varepsilon} \langle{k}^{-1} \rangle^{\frac{1}{6} - 2\varepsilon}$, applying Lemma \ref{pre1} yields
	    \begin{align*}
	    \mathcal{N}^{y,2}_{3,1} &= \left| c_{\beta} \int_{\R^2}\mathbf{1}_{ 2\nu\le2|k| \le |k-\ell|  }  e^{2c \lambda_k t} \nu^{\frac{1}{3}} |k|^{-\frac{4}{3}} \langle k \rangle^{2m} \langle{k}^{-1} \rangle^{2\varepsilon}   \operatorname{Re} \langle ik   \omega_k,   \left( \partial_{y} \phi_{\ell} \cdot i(k-\ell) \partial_{y}\omega_{k-\ell}\right) \rangle   \,{dk} \,{d\ell} \right| \\
	    &\lesssim \left\| \langle{k}^{-1} \rangle^{\frac{1}{6}} |k|^{-\frac{1}{3}} \| \omega_{k} \|_{L^2_y} \right\|_{L^1_k} \nu^{\frac{1}{6}} \left\| \langle \ell \rangle^{m}  \langle{\ell}^{-1} \rangle^{\varepsilon} |\ell|^{\frac{5}{6}} \| \partial_{y} \phi_{\ell} \|_{L^\infty_y} \right\|_{L^2_\ell} \nn\\
&\qquad\times\nu^{\frac{1}{6}} \left\| \langle k \rangle^{m}  \langle{k}^{-1} \rangle^{\varepsilon} |k|^{\frac{1}{3}} \| \partial_{y} \omega_{k} \|_{L^2_y} \right\|_{L^2_{\{|k| \ge \nu\}}}  \\
	    &\lesssim \nu^{-\frac{1}{3}} \mathscr{E}^{\frac{1}{2}} \mathscr{D}_1^{\frac{1}{4}} \mathscr{D}_2^{\frac{1}{4}} \mathscr{D}_3^{\frac{1}{2}},
	    \end{align*}
	    the last line utilizes the facts that $ |\ell|^{\frac{5}{6}} \| \partial_{y} \phi_{\ell} \|_{L^\infty_y } \lesssim |\ell|^{\frac{1}{3}} \| \nabla_\ell \partial_{y} \phi_{\ell} \|_{L^2_y} \lesssim |\ell|^{\frac{1}{3}} \| \omega_{\ell} \|_{L^2_y} $ and
$$ \nu^{\frac{1}{6}} \left\| \langle k \rangle^{m}  \langle{k}^{-1} \rangle^{\varepsilon} |k|^{\frac{1}{3}} \| \partial_{y} \omega_{k} \|_{L^2_y} \right\|_{L^2_{\{|k| \ge \nu\}}} \lesssim \nu^{-\frac{1}{3}} \mathscr{D}_1^{\frac{1}{4}} \mathscr{D}_2^{\frac{1}{4}}. $$
	    Therefore, combining all the estimates for $ \mathcal{N}^{y,1}_{3,1}$ and $ \mathcal{N}^{y,2}_{3,1}$ across the various regions, we conclude that
	    \begin{align}\label{N31y}
	    	\mathcal{N}^{y}_{3,1} \lesssim \nu^{-\frac{1}{2}} \mathscr{E}^{\frac{1}{2}} \big( \mathscr{D}_1^{\frac{3}{8}} \mathscr{D}_3^{\frac{5}{8}} +   \mathscr{D}_1^{\frac{1}{2}} \mathscr{D}_3^{\frac{1}{2}} +  \mathscr{D}_1^{\frac{3}{4}} \mathscr{D}_3^{\frac{1}{4}} +  \mathscr{D}_1^{\frac{1}{8}} \mathscr{D}_2^{\frac{1}{4}} \mathscr{D}_3^{\frac{3}{8}} \mathscr{D}_5^{\frac{1}{4}} +  \mathscr{D}_1^{\frac{1}{4}} \mathscr{D}_2^{\frac{1}{4}} \mathscr{D}_3^{\frac{1}{2}}\big).
	    \end{align}
	    This completes the proof of the lemma by combining \eqref{N31x} and \eqref{N31y}.
	\end{proof}
\end{lemma}

\subsection{Completion the proof of the nonlinear stability of Theorem \ref{thm1.2}}
In this part of the proof, we make use of the bootstrap argument to finalize the argument of Theorem \ref{thm1.2}. As a first step, we introduce the following energy functional
\begin{align*}
	 \mathscr{E}_{\text{total}}(t) \stackrel{\mathrm{def}}{=} \sup_{0 \le \tau \le  t}\mathscr{E}(\tau) + \int_{0}^{t} \mathscr{D}(\tau) {d} \tau.
\end{align*}
Then, invoking Lemmas \ref{le_1}-\ref{le_3} and combining with \eqref{energy}, we arrive at
\begin{align}
	\mathscr{E}_{ \text{total}}(t) \lesssim \mathscr{E}_{ \text{total}}(0) +  \nu^{-\frac{1}{2}} \mathscr{E}_{ \text{total}}^{\frac{3}{2}}(t). \label{Ew_total}
\end{align}
Furthermore, taking into account the initial condition
\begin{align*}%\label{initial_condition}
\sum_{0\le j\le 1 } \left\| (\nu^{\frac{1}{3} } \partial_{y})^j \langle \partial_x\rangle^{m-\frac{j}{3}} \left\langle \frac{1}{\partial_x} \right\rangle^{\varepsilon} \omega_{in} \right\|_{L^2_{x,y}} \le \varepsilon_0   \nu^{\frac{1}{2} },
\end{align*}
it is enough to select $ \varepsilon_0$ small enough. Under the bootstrap hypothesis, \eqref{Ew_total} then immediately yields
\begin{align*}
	\mathscr{E}_{ \text{total}}(t) \le C \varepsilon_0 \nu.
\end{align*}
And thus the proof of Theorem  \ref{thm1.2} is concluded. \hspace{7.8cm}
$\square$

\bigskip

\vskip .2in
\section*{Acknowledgement}
\noindent{Wu was partially supported by the National Science Foundation of the United States under Grants DMS 2104682 and DMS 2309748. Zhai was partially supported by  the Guangdong Provincial Natural Science Foundation under grant 2024A1515030115. }

 \vskip .2in
\noindent{\bf Data Availability Statement} Data sharing is not applicable to this article as no
data sets were generated or analysed during the current study.

\vskip .2in

\noindent{\bf Conflict of Interest} The authors declare that they have no conflict of interest. The
authors also declare that this manuscript has not been previously published, and
will not be submitted elsewhere before your decision.

\end{document}